\documentclass[titlepage]{amsart}
\usepackage[draft]{marktext}
\usepackage{gimac}
\usepackage[foot]{amsaddr}

\newif\ifdraft
\makeatletter\ifGI@draft\drafttrue\else\draftfalse\fi\makeatother

\usepackage{bbm}
\ifdraft
  \usepackage{datetime}\usdate
  \usepackage{currfile} 
\fi

\RequirePackage{tcolorbox}
\tcbuselibrary{theorems}
\renewcommand{\boxed}[1]{\tcboxmath[boxrule=.5pt,colback=cyan!5]{#1}}
\newcommand{\eqdef}{\defeq}

\newcommand{\hide}[1]{}  

\newcommand{\bnu}{\bar\nu}

\WithSuffix\newcommand\cite*[2]{\cite[#2]{#1}}

\newcommand{\inv}{^{-1}}

\newcommand{\vp}{\varphi}

\newcommand{\nwc}{\newcommand} 
\nwc{\rplus}{{\R^+}}
\nwc{\nplus}{{\N_0}} 

\nwc{\supj}{^{(j)}}
\nwc{\supk}{^{(k)}}
\nwc{\supn}{^{(n)}}

\nwc{\suph}[1]{{#1}^h}
\nwc{\suphk}[1]{{#1}^{h_k}}
\nwc{\subhk}[1]{{#1}_{h_k}}
\nwc{\supht}[1]{{#1}^{h,\tau}}
\nwc{\subhtk}[1]{{#1}_{h_k,\tau_k}}
\nwc{\subht}[1]{{#1}_{h,\tau}}
\nwc{\scalek}[1]{\breve{#1}_k}

\DeclareBoldMathCommand{\one}{\mathbbm{1}}
\DeclareMathOperator{\re}{\rm Re}

\renewcommand{\ge}{\geqslant}

\begin{document}
  \title[\ifdraft DRAFT \mmddyyyydate\today\ \currenttime\ \fi Coagulation and universal scaling limits]
{Coagulation and universal scaling limits for critical Galton-Watson processes}
  \dedicatory{Dedicated to the memory of Jack Carr.}
  \author[Iyer]{Gautam Iyer\textsuperscript{1}}
  \address{%
    \textsuperscript{1} Department of Mathematical Sciences,
    Carnegie Mellon University,
    Pittsburgh, PA 15213.}
  \email{gautam@math.cmu.edu}

  \author[Leger]{Nicholas Leger\textsuperscript{2}}
  \address{
    \textsuperscript{2} Department of Mathematics,
    University of Houston,
    4800 Calhoun Rd.,
    Houston, TX 77004.}
  \email{nleger@math.uh.edu}

  \author[Pego]{Robert L. Pego\textsuperscript{1}}
  \email{rpego@cmu.edu}
  \thanks{
    This material is based upon work partially supported by
    the National Science Foundation under grants
    the National Science Foundation (through grants
    DMS-1252912 to GI,
    DMS-1515400 to RP),
    the Simons Foundation (through grants
    \#393685 to GI,
    \#395796 to RP),
    and the Center for Nonlinear Analysis (through grant NSF OISE-0967140).
  }
  \subjclass[2010]{Primary 60J80; Secondary 37E20, 35Q70}
  \keywords{Smoluchowski equation, multiple coalescence, Galton-Watson processes, CSBPs, universal laws, Bernstein transform}

  \begin{abstract}
    The basis of this paper is the elementary observation that the $n$-step descendant distribution of any Galton-Watson process satisfies a discrete Smoluchowski coagulation equation with multiple coalescence.
Using this we obtain simple necessary and sufficient criteria for the convergence of scaling limits of critical Galton-Watson processes in terms of scaled family-size distributions and a natural notion of convergence of L\'evy triples.
    Our results provide a clear and natural interpretation, and an alternate proof, of the fact that the L\'evy jump measure of certain CSBPs satisfies a generalized Smoluchowski equation.
    (This result was previously proved by Bertoin and Le Gall in 2006.)

    Our analysis shows that the nonlinear scaling dynamics of CSBPs becomes \emph{linear and purely dilatational} when expressed in terms of the L\'evy triple associated with the branching mechanism.
    We prove a continuity theorem for CSBPs in terms of the associated L\'evy triples,
    and use our
    scaling analysis to prove existence of \emph{universal} critical Galton-Watson and CSBPs analogous to W. Doeblin's ``universal laws''.
    Namely, these universal processes generate \emph{all} possible critical and subcritical CSBPs as subsequential scaling limits.

    Our convergence results rely on a natural topology for L\'evy triples and a continuity theorem for Bernstein transforms (Laplace exponents) which we develop in a self-contained appendix.
  \end{abstract}
  \maketitle
  \tableofcontents
  \newpage

  \section{Introduction.}\label{s:Intro}
In 1875, Galton and Watson~\cite{WatsonGalton75} took up an investigation into
the phenomenon of ``the decay of the families of men who occupied conspicuous
positions in past times.'' The problem, posed by Galton, was summarized by the 
Rev. H. W. Watson as follows.
\begin{quote}
``Suppose that at any instant all the adult males of a large nation have
different surnames, it is required to find how many of these surnames will have
disappeared in a given number of generations upon any hypothesis, to be determined
by statistical investigations, of the law of male population.''
\end{quote}
Their analysis led to the eventual creation of a class of time-discrete, size-discrete branching processes now called Galton-Watson processes or Bienaym\'e-Galton-Watson processes, 
 since key aspects of the topic were discovered thirty years earlier by I. J. Bienaym\'e (see~\cite{HeydeSeneta77,Bacaer11}).
  The study of these and other branching processes has led to numerous interesting lines of research, and we refer the reader to~\cite{AthreyaNey04} where the basic theory and applications are beautifully laid out.

Branching and coalescence are intrinsically connected concepts, and are in some sense dual.
Indeed, following a tree from root to leaf leads to branching, and following it from leaf to root leads to coalescence.
Mathematically, a deep connection between branching and coalescence was exposed by Kingman~\cite{Kingman82,Kingman82a} 
in his efforts to address questions of ancestry in population genetics.
Loosely speaking, ancestral lineages branch as time proceeds forward, 
while the Kingman coalescent traces the sizes of groups or clusters of individuals 
with a common living ancestor as time evolves backwards.
Kingman was particularly concerned with models that emerged from early work of
Moran and Wright that hold total population fixed.
His work was subsequently expanded to study a general class of exchangeable coalescent models,
and a natural completion---the class of $\Lambda$-coalescents---was identified and studied by
Pitman~\cite{Pitman99a} and Sagitov~\cite{Sagitov99}.
We refer to the book \cite{Bertoin06} and the survey papers \cite{Berestycki09,GnedinIksanovEA14}
for an overview of the now-extensive literature in this area. 
For aspects particularly related to CSBPs, we refer 
to~\cite{
    BertoinLeGall00,
    Lambert03,
    Schweinsberg03,
    BirknerBlathEA05,
    BerestyckiBerestyckiEA08,
    BerestyckiBerestyckiEA14
  }.
%
  

The present work focuses on classical Galton-Watson processes and their 
continuum limits---continuous state branching processes (CSBPs)---which 
are models in which populations may fluctuate.
Our work was originally motivated by observations of Bertoin and Le Gall
contained in a series of papers that relate certain coalescence models to
CSBPs~\cite{BertoinLeGall03,BertoinLeGall05,BertoinLeGall06}.
In \cite{BertoinLeGall06} the authors prove a striking result
which shows that the L\'evy jump measure associated to certain CSBPs satisfies
a generalized Smoluchowski coagulation equation with multiple coalescence.  
One motivation for this work was to provide a simple and natural explanation for this fact.

Our starting point 
is an elementary observation showing that the $n$-step descendant distribution of the Galton-Watson process satisfies a discrete Smoluchowski coagulation equation with multiple coalescence.
We use this basic connection to establish simple criteria for the convergence of scaling limits of critical Galton-Watson processes and study the coagulation dynamics of the L\'evy jump measure of certain CSBPs. (As we consider only classical time-homogeneous Galton-Watson processes, we make no connection with the large literature on branching in varying and random environments, except for recent work of Bansaye and Simatos~\cite{BansayeSimatos15} where the authors develop general convergence criteria that compare closely with ours.)

Our scaling and limit analysis of the coagulation dynamics associated with critical Galton-Watson processes leads to a simple yet striking observation:
The \emph{nonlinear scaling dynamics of CSBPs becomes linear and purely dilational} when expressed in terms of the L\'evy triple that represents the branching mechanism.
This observation is parallel to results of \cite{MenonPego08}
that concern the dynamic scaling analysis of solvable Smoluchowski equations, 
and extends an analogy between CSBPs and infinite divisibility which has been
evident since the work of Grimvall \cite{Grimvall74} classifying all continuum limits
of Galton-Watson processes as CSBPs.
Here we make use of the dilational representation to establish the existence
of certain \emph{universal} critical Galton-Watson processes and CSBPs which generate
\emph{all} critical and subcritical CSBPs as subsequential scaling limits.
These universal processes are analogous to W. Doeblin's ``universal laws'' 
in classical probability theory \cite[XVII.9]{Feller71}.
Further, they supply
a precise interpretation to a remark made by Grey~\cite{Grey74} to the effect that a large class of ``critical and subcritical processes $\ldots$ do not seem to lend themselves to suitable scaling'' which yields a well-defined limit.

  \subsection*{Summary of results.}
  The principal new results we prove in this paper include the existence of universal Galton-Watson processes, and the existence of universal CSBPs.

  \begin{theorem}\label{t:IntroUnivGW}
    There exists a critical Galton-Watson process $X_\star$ such that \emph{any} (sub)\-critical CSBP that remains finite almost surely can be obtained as a subsequential scaling limit of $X_\star$.
  \end{theorem}

  \begin{theorem}\label{t:IntroUnivCSBP}
    There exists a critical CSBP $Z_\star$ such that \emph{any} (sub)critical CSBP that remains finite almost surely can be obtained as a subsequential scaling limit of~$Z_\star$.
  \end{theorem}

These results follow naturally and directly from our analysis of discrete coalescence models in terms of the topology of L\'evy triples that we introduce.
The proof of Theorem~\ref{t:IntroUnivCSBP} in particular is based upon a new continuity theorem (Theorem~\ref{t:CSBPCont}) that shows that
convergence of (sub)critical CSBPs is equivalent to convergence of the L\'evy triples which generate them. 
Our paper also contains a variety of other results concerning these issues.
In the rest of this section, we summarize the main results and organization of the paper.

  \subsubsection*{Time and size discrete coalescence and branching.}

  Part~\ref{p:disc} (Sections~\ref{s:DCoag}--\ref{s:BernsteinEvol}) is devoted to the study of discrete-time, discrete-size coalescence and branching.
  We begin by describing (Section~\ref{s:DCoag}) a time-discrete Markov process~$C$ modeling sizes of clusters undergoing coalescence. 
  (A more general class of processes of the type we study has been introduced independently in \cite{GrosjeanHuillet16}, as discussed in section~\ref{s:DCoag} below.)
  We show (Proposition~\ref{p:PnHat}) that the backward equation of the process~$C$ is exactly a time-discrete analog of the well known Smoluchowski coagulation equation~\cite{Smoluchowski16,Smoluchowski17}, a mean-field rate-equation model of clustering.
  Moreover we show that the one dimensional distributions of coordinates of~$C$ have the same distribution as a Galton-Watson process (see \cite{GrosjeanHuillet16} for more extensive results concerning genealogies).
  This provides a elementary and fundamental connection between branching and coalescence, and shows 
  that the $n^\text{th}$ generation descendant distribution of a critical Galton-Watson process is itself a solution of the discrete Smoluchowski equation with multiple coalescence.
  This will be used extensively in the scaling analysis performed in Part~\ref{p:scaling}, below.

  We conclude our treatment of the discrete scenario by introducing the \emph{branching mechanism} and the \emph{Bernstein transform} for Galton-Watson processes.
  We show (Proposition~\ref{p:BtEvol}) that
  the discrete coagulation dynamics gives an elegant (discrete-time) evolution equation for the Bernstein transform in terms of the branching mechanism.
  This parallels the evolution of the Bernstein transform of the size distribution in the Smoluchowski dynamics (see~\cite{BertoinLeGall06,IyerLegerEA15} or equation~\eqref{e:PhiEvol1}, below), and is a direct analog of the backward equation for critical CSBPs that become extinct almost surely.

  \subsubsection*{Convergence criteria for scaling limits of critical Galton-Watson processes.}
  Part~\ref{p:scaling} (Sections~\ref{s:CSBPintro}--\ref{s:Grimvall}) concerns the study of scaling limits of critical Galton-Watson processes.
  We begin with a brief introduction to CSBPs (Section~\ref{s:CSBPintro}) and discuss Bertoin and Le Gall's striking result from \cite{BertoinLeGall06} showing that the {L\'evy jump measure} of certain CSBPs satisfies a special form of the Smoluchowski equation with multiple coalescence.

  Using improvements in the theory of Bernstein functions \cite{SchillingSongEA10} (known in probability as Laplace exponents), we obtain simple and precise criteria for the existence of scaling limits of critical Galton-Watson processes expressed directly in terms of L\'evy convergence of rescaled reproduction laws (Propositions~\ref{p:BrMechConv}, \ref{p:BTConv} and~\ref{p:GWConv}).
  These convergence results provide a clear description of how the L\'evy jump measure of certain CSBPs arises as the rescaled limit of $n$-step descendant distributions of Galton-Watson processes, and explains how the generalized Smoluchowski equation arises naturally in~\cite{BertoinLeGall06}.
  The proofs make use of a continuity theorem for Bernstein transforms which appears to be little known.  As we expect this to be of wider utility, we develop this theory 
separately\footnote{
    The well known continuity theorems guarantee that weak-$\star$ convergence of probability measures is equivalent to pointwise convergence of their characteristic functions (or Laplace transforms).
    In our context, the continuity theorem shows that convergence of L\'evy triples is equivalent to pointwise convergence of the Bernstein transforms.
    The proof is similar to that of the classical continuity theorems. However, since it is not readily available in the literature we prove it in Appendix~\ref{s:CT} (Theorem~\ref{t:Cont}).
  }
 in Appendix~\ref{s:CT}.

  Finally, the convergence of scaling limits of general Galton-Watson processes is a classical subject and has been studied by many authors.
In section~\ref{s:Grimvall} we compare our criteria to the classical
convergence criteria provided by Grimvall~\cite{Grimvall74} and the
simplified criteria recently derived by Bansaye and Simatos \cite{BansayeSimatos15}
in connection with their more general investigation of scaling limits for
Galton-Watson processes in varying environments.  
We establish (Proposition~\ref{p:GrimvallEquiv}) equivalence
between these criteria and L\'evy triple convergence in the critical case.

  \subsubsection*{Universality in Galton-Watson processes and CSBPs.}

  Part~\ref{p:univ} (Sections~\ref{s:universalGW}--\ref{s:LinUniv}) is devoted to studying universality and proving Theorems~\ref{t:IntroUnivGW} and \ref{t:IntroUnivCSBP}.
  For Galton-Watson processes, we use our simplified convergence results to show  that there exists a Galton-Watson process with a ``universal'' family size distribution (Theorem~\ref{t:UniversalGW}).

The proof that the scaling dynamics of CSBPs becomes {linear and purely dilatational} when expressed in terms of the L\'evy triple associated with the branching mechanism 
comes down to a simple calculation (Proposition~\ref{p:Lin}).
We then show that the map from L\'evy triple to solution is bicontinuous in a particular sense, 
proving a continuity theorem for CSBPs (Theorem~\ref{t:CSBPCont}).
  From this and our study of dilational dynamics for L\'evy triples in section~\ref{s:universalGW}, we infer the existence of universal CSBPs whose subsequential scaling limits yield all possible critical and subcritical CSBPs (Theorem~\ref{t:IntroUnivCSBP} above, restated precisely as Theorem~\ref{t:UniversalCSBP}).

%

  \part{Time and size discrete coalescence and branching.}\label{p:disc}

  \section{Coagulation equations and processes with multiple mergers.}\label{s:DCoag}

The Smoluchowski coagulation equation~\cite{Smoluchowski16,Smoluchowski17}  
is a mean-field rate-equation model of coalescence that may be used to describe 
a system of clusters 
that merge as time evolves to form larger clusters.
{The type of objects that comprise the clusters varies widely 
in applications---examples include smog particles, animals, and dark matter.}
Smoluchowski's equation governs the time evolution of the \emph{cluster-size distribution},
under certain assumptions which usually include the assumption that only 
binary mergers are taken into account. 
 Our aim in this section is to describe a natural and elementary Markov process that directly relates to a time-discrete Smoluchowski equation generalized to account for simultaneous mergers of any number of clusters. 

     The coalescence processes that we describe bear a close
 relation to Galton-Watson processes that will be delineated in the following section.
As indicated in the introduction, the coalescence process we describe is
a time-homogeneous version of a class of processes introduced independently
by Grosjean and Huillet in \cite{GrosjeanHuillet16}. 
The paper \cite{GrosjeanHuillet16} develops more extensively a correspondence between
coalescence and branching processes in terms of genealogies and ancestral trees.
Our use of the coalescence process in this paper is quite different, 
simply to indicate a clear relation between branching and the Smoluchowski
equation, as preparation for studying scaling limits.


%
%
  \subsection{A discrete-time process modeling coalescence.}\label{s:Cprocess}
  Let $\hat \pi$ be a probability measure on $\nplus\defeq \N\cup\set{0}$.
  We consider a countably infinite collection of clusters, each with nonnegative integer size, which undergo coalescence at discrete time steps according to the following {rules}.
    \begin{enumerate} 
   \item\label{i:C1}
   At each time step, 
{the existing clusters are collected into disjoint groups (which may be empty), and
the clusters in each group merge to form the new collection of clusters.
Empty groups create clusters of zero size.}
    \item\label{i:C2}
      The probability that a merger involves exactly $k$ clusters is $\hat \pi(k)$.  We refer to $\hat \pi$ as the \emph{merger distribution}.    

    \item\label{i:C3}
      The sizes of each of the $k$ clusters participating in a simultaneous $k$-merger are independent and identically distributed.
  \end{enumerate}
 An essential feature of this process is that the distribution of cluster sizes after a simultaneous $k$-merger is the same as the distribution of the sum of $k$ independent copies of 
  the cluster sizes before the merger.

An elementary construction of a coagulation process $C=(C_n)_{n\in\nplus}$ that has the above features proceeds as follows.
For each time $n \in \nplus$, $C_n$ is to be a random function from $\N$ to $\nplus$,
whose values $C_n(j)$ represent the size of a cluster in the ensemble.
At the initial time $n=0$, the random variables $C_0(j)$, ${j \in \N}$, may in
general be taken as independent and identically distributed; however, for the sake
of comparison with Galton-Watson processes, here we will always assume
\begin{equation*} 
  C_0(j) = 1\qquad\text{for all $j\in\N$} \,.
\end{equation*}

Given $C_n$ we define $C_{n+1}$ as follows.
Choose a random sequence $(M_{k,n})_{k\in\N_0}$ independent of $C_m$ for $0\leq m\leq n$, 
such that $M_{0,n} = 0$ and the increments 
\[
M_{k+1,n} - M_{k,n}\in \nplus
\]
are independent and identically distributed with law $\hat \pi$.
Define 
  \begin{equation}\label{e:Cdef}
    C_{n+1}(j) = \sum_{i = 1+ M_{j-1,n}}^{M_{j,n}} C_n(i) \,,
  \end{equation}
  with the convention that the sum is $0$ if $ M_{j,n}=M_{j-1,n}$ (this corresponds to the creation of a new cluster of size zero).
  The quantity $M_{j+1,n} - M_{j,n}$ represents the number of clusters that simultaneously merge  and combine their sizes to form a new cluster. 

Note that $C$ is a Markov process.  Moreover, for each $n \geq 0$, the random variables ${C_n(i)}$, ${i \in \N}$, are independent and identically distributed.
Thus~\eqref{e:Cdef} guarantees that for any $j \in \N$, the random variable $C_{n+1}(j)$ is the sum of $N$ independent and identically distributed copies of $C_{n}(j)$, where $N$ is itself a random variable with distribution $\hat \pi$ and independent of $C_n$.
This shows that the process $C$ meets conditions~\eqref{i:C1}--\eqref{i:C3} above.

We remark, however, that while the sequence $(C_n(1), C_n(2), \dots)$ represents the sizes of all clusters in the system, the individual coordinate functions $n\mapsto C_n(j)$ do not track the time evolution of the size of a particular cluster and need not form a Markov process.
  
\begin{remark*}
In case $\hat\pi(0)=0$, the partial sums $(\sum_{j=1}^k C_n(j))_{k\in\N}$
correspond to the marks of a \emph{renewal process} with integer increments,
for each $n\geq0$.
The coagulation process $C$ corresponds to 
a discrete-time type of \emph{thinning} of these processes---marks disappear according to the rules that govern the process $C$. 
It was pointed out by Aldous in \cite{Aldous99} that independent thinning of renewal processes yields a classical Smoluchowski coagulation equation in continuous time.
We find the condition $\hat\pi(0)>0$ necessary, however, in order to produce the correspondence with critical Galton-Watson processes to appear in Section~\ref{s:DCandGW} below.
\end{remark*}

  \subsection{Evolution of the cluster-size distribution.}

We let $\nu_n$ denote the distribution of cluster sizes at step $n$
in the coagulation process above. 
  That is,  $\nu_n(j)$ denotes the chance that any single cluster has size $j$ after $n$ time steps.
  Because the variables $\set{C_n(i)}_{i \in \N}$ are all identically distributed, we  have
  \begin{equation}\label{e:PnHatDef}
      \nu_n(j) = \P\set{ C_n(i) = j} 
    \quad\text{for all } i \in \N \,.
  \end{equation}

  Our first observation is that the $n$-step size distribution $\nu_n$ determines $\nu_{n+1}$ in a manner that naturally captures~(\ref{i:C1})--(\ref{i:C3}).

  \begin{proposition}\label{p:evol}
    For any $n \geq 0$ and $j \in \N$ we have
    \begin{equation}\label{e:PnHatEvol}
      \nu_{n+1}(j) = \sum_{k \geq 0} \hat \pi(k) \nu_n^{*k}(j) \,.
    \end{equation}
  \end{proposition}
    Here $\nu_n^{*k}$ denotes the $k^\text{th}$ convolution power of $\nu_n$ and is given by
    \begin{equation}\label{e:PnHatStarK}
      \nu_n^{*k}(j)
	= \sum_{i_1, \dots, i_k \geq 0}
	    \delta^j( i_1 + \cdots + i_k)
	    \nu_n(i_1) \cdots \nu_n(i_k) \,,
    \end{equation}
    where $\delta^j$ denotes a Kronecker delta function:%
    \begin{equation*}
      \delta^j(k) =
      \begin{cases}
	1, & j = k \,,\\
	0, & j \neq k \,.
      \end{cases}
    \end{equation*}
By convention, we define $\nu_n^{*0}(j) = \delta^0(j)$.
  \begin{proof}[Proof of Proposition~\ref{p:evol}]
    The formula follows directly from~\eqref{e:Cdef}, independence of the $C_n(i)$ ($i \in \N$), and the fact that $M_{k+1,n} - M_{k,n}$ has distribution $\hat \pi$.
  \end{proof}

  \begin{remark*}
    In light of the assumption $C_0\equiv1$, the size distribution after one time step is exactly the merger distribution (i.e.\ $\nu_1 = \hat \pi$).
    This follows immediately from~\eqref{e:PnHatEvol}.
  \end{remark*}

  \subsection{The discrete Smoluchowski equation.}

  The size distribution of the process~$C$ naturally relates to a time-discrete Smoluchowski equation with multiple coalescence, as we now explain.
  The key idea is to express the dynamics in terms of clusters of non-zero size.
  For convenience, we will focus on case when the merger distribution is critical, 
  according to the following terminology.
  \begin{definition}
      We say the merger distribution $\hat \pi$ is \emph{critical} if
    \begin{equation}\label{e:Criticality}
      \sum_{k \geq 0} k \hat \pi(k) = 1 \,.
    \end{equation}
    This means the expected number of clusters involved in a simultaneous merger is $1$.
  \end{definition}


  \begin{proposition}\label{p:PnHat}
    If the merger distribution $\hat \pi$ is critical, then the size distribution~$\nu_n$ satisfies 
    \begin{align}\label{e:PnHatSmol}
      \nu_{n+1}(j) - \nu_n(j)
	&= \sum_{k \geq 2} \hat R_k( \rho_n ) \hat I_k( \nu_n, j ), \qquad j>0 \,,
    \intertext{where}
    \label{e:RhoN}
      \rho_n
      &\defeq \P\set{ C_n(1) > 0 }
	= \sum_{j>0} \nu_n(j)\,,
    \\
      \label{e:RkHat}
      \hat R_k(\rho_n) &\defeq
	\sum_{l \geq k} \hat \pi(l) \binom{l}{k} \rho_n^k (1 - \rho_n)^{l-k}\,,
    \\
    \label{e:IkHat}
      \hat I_k( \nu_n, j )
	&\defeq \sum_{i_1, \dots, i_k > 0}
	    \brak[\Big]{
	      \delta^j \paren[\Big]{ \sum_{l = 1}^k i_l }
	      - \sum_{l = 1}^k \delta^j \paren{ i_l }
	    }
	    \frac{\nu_n(i_1)}{\rho_n}
	    \cdots
	    \frac{\nu_n(i_k)}{\rho_n} \,.
    \end{align}
  \end{proposition}

  \begin{remark*}
    Equations~\eqref{e:PnHatSmol}--\eqref{e:IkHat} are a discrete version of the Smoluchowski equations (see Section~\ref{s:CSBPintro}, below) and can be interpreted as follows.
    Observe that at time~$n$, the chance that any single cluster has positive size is exactly~$\rho_n$.
    Consequently the chance that any simultaneous merger involves exactly~$k$ clusters of non-zero size is given by $\hat R_k(\rho_n)$.
    Now in the event of a simultaneous merger of $k$ non-zero sized clusters, the chance that a cluster of size $j$ is created is exactly the positive term in~\eqref{e:IkHat}.
    On the other hand, a cluster of size $j$ can itself be involved in this simultaneous merger, leading to the creation of a larger cluster and the destruction of the cluster of size~$j$.
    The chance that this happens is the negative term in~\eqref{e:IkHat}.
    Consequently, $\hat I_k( \nu_n, j)$ is the expected change in frequency of clusters of size $j$, given that a simultaneous merger of $k$ non-zero sized clusters occurred.
   Assuming independence and summing yields~\eqref{e:PnHatSmol}.
  \end{remark*}

  \begin{proof}[Proof of Proposition~\ref{p:PnHat}]
    Equation~\eqref{e:PnHatSmol} follows when we restrict the size dynamics to clusters of non-zero size.
    Indeed, $\nu_n(j) / \rho_n$ is exactly the chance that a cluster has size $j$, conditioned on having non-zero size.
    Since $\hat R_k$ is the chance that exactly~$k$ such clusters merge, we must have
    \begin{align}
      \nonumber
      \nu_{n+1}(j)
	&= \sum_{k \geq 1}
	    \hat R_k( \rho_n)
            \paren[\Big]{ \frac{\nu_n}{\rho_n} }^{*k}(j)
      \\
      \label{e:PnHatNplus1}
	&= \sum_{k \geq 1}
	    \hat R_k( \rho_n)
	    \sum_{i_1, \dots, i_k > 0}
	      \delta^j(i_1 + \cdots +i_k)
		\frac{\nu_n(i_1)}{\rho_n} \cdots \frac{\nu_n(i_k)}{\rho_n} \,.
    \end{align}

Next we compute the expected number of clusters of size $j > 0$ involved in a given merger.
On one hand, this should be exactly the expected number of clusters involved in a merger, times $\nu_n(j)$.
On the other, this can be computed by counting the number of clusters of size $j$ in every cluster merger involving at least one cluster with non-zero size.
    This gives
    \begin{multline}\label{e:PnHatN}
       \paren[\Big]{ \sum_{k \geq 0} k \hat \pi(k) }\nu_n(j)
      \\
	= \sum_{k \geq 1}
	    \hat R_k(\rho_n)
	    \sum_{i_1, \dots, i_k > 0}
	      \paren[\big]{\delta^j(i_1) + \cdots + \delta^j(i_k)}
	      \frac{\nu_n(i_1)}{\rho_n} \cdots \frac{\nu_n(i_k)}{\rho_n} \,.
    \end{multline}
    Using the criticality assumption~\eqref{e:Criticality}, subtracting~\eqref{e:PnHatN} from~\eqref{e:PnHatNplus1} and observing that the~$k = 1$ term cancels, we obtain~\eqref{e:PnHatSmol}.
  \end{proof}

  \begin{remark*}
    Without the criticality assumption~\eqref{e:Criticality}, the above shows that~\eqref{e:PnHatSmol} should be replaced with
    \begin{equation}\label{e:PnHatSmolSubC}
      \nu_{n+1}(j) - \nu_n(j)
	= (\Xi - 1) \nu_n(j)
	  + \sum_{k \geq 2} \hat R_k( \rho_n ) \hat I_k( \nu_n, j ) \,, \qquad j>0 \,,
    \end{equation}
    if
    \begin{equation}\label{d:Xi}
      \Xi \defeq \sum_{k \geq 0} k \hat \pi(k)\,,
    \end{equation}
    the expected number of clusters involved in a simultaneous merger, is finite.
  \end{remark*}

  \section{Galton-Watson processes and coagulation dynamics.}\label{s:DCandGW}

  The purpose of this section is to establish a direct and elementary connection between critical Galton-Watson processes and the coalescence processes described in Section~\ref{s:DCoag}.

  \subsection{A brief introduction to Galton-Watson processes}

  The Galton-Watson process was introduced to model the
  gradual extinction of Victorian aristocratic family names, despite an increase in the general population.
  The original model supposes that family names are passed down only through male heirs, and that each heir reproduces independently and identically, and is replaced by his sons in the subsequent generation.

  To fix notation, let $\hat \pi$ denote the family-size distribution, meaning $\hat\pi(j)$ is the chance that each heir has exactly $j$ sons, and let $X_n$ denote the number of male descendants of a single individual after $n$ generations.
  Presuming the descendants reproduce independently, the process $X=\set{X_n}$ is a time-homogeneous Markov process taking values in $\nplus$,
  with $X_0=1$.  Further, the $n$-step transition probabilities
  \begin{equation*}
    P_n(i, j) \defeq \P\set[\big]{ X_n = j \given X_0 = i }
  \end{equation*}
   satisfy the convolution property
  \begin{equation*}
    P_n(i_1 + i_2, j) = \sum_{k = 0}^j P_n( i_1, j-k ) P_n(i_2, k ) \,.
  \end{equation*}
    The standard construction of a Galton-Watson process is to choose an array $\set{ \xi_{i,j} }$ of independent, identically distributed random variables with distribution $\hat \pi$, and define
  \begin{equation}\label{e:Xevol}
    X_{n+1} = \sum_{k = 1}^{X_n} \xi_{k,n+1} \,.
  \end{equation}
  We refer the reader to~\cite{AthreyaNey04} for a more complete introduction.

\begin{remark*}
The formula \eqref{e:Cdef} appears dual to \eqref{e:Xevol} in a curious way:
The number of terms in the sum in \eqref{e:Cdef} is distributed 
according to $\hat\pi$, while it is the individual terms in \eqref{e:Xevol}
that are so distributed. 
The individual terms in \eqref{e:Cdef} are generated by the accumulation
formula, while it is the number of terms in \eqref{e:Xevol} that is so
generated.  
\end{remark*}

%
%
  \subsection{Discrete coagulation dynamics of the descendant distribution.}
  
Next we establish a direct, explicit and elementary connection between the Galton-Watson process $X$ and the coagulation process $C$.
Namely, we show that $X_n$, the number of descendants of a single individual after $n$ generations, has the same distribution as the cluster sizes produced by the process $C$ after $n$ steps.
More extensive connections between genealogies and ancestral trees of these processes 
(and related time-inhomogenous processes) have recently been established in \cite{GrosjeanHuillet16}.

  \begin{proposition}\label{p:PnEvol}
Let $X$ be a Galton-Watson process, and let
 \begin{equation}\label{e:PnPiHat}
      \hat\nu_n(j) \defeq  \P\set{X_n = j} = P_n(1, j) \,.
    \end{equation}
denote the distribution of $X_n$.
Suppose $C$ is a coagulation process as in Section~\ref{s:Cprocess} having merger distribution $\hat\pi$ equal to the family-size distribution of $X$.
Then for each $n\geq 0$, $X_n$ has the same distribution as $C_n(i)$, $i\in\N$. 
That is, 
\[
\hat\nu_n = \nu_n
\]
for all $n \geq 0$,
where $\nu_n$ is defined in~\eqref{e:PnHatDef}.
Consequently, if $X$ is critical, then $\hat\nu_n$ satisfies the discrete Smoluchowski equations~\eqref{e:PnHatSmol}--\eqref{e:IkHat}.
%
  \end{proposition}
  
  Recall a Galton-Watson process $\set{X_n}$ is said to be \emph{critical} if the expected number of descendants of a single individual is $1$.
  Explicitly, this means that~$\hat \pi$ satisfies~\eqref{e:Criticality},
  and it is well known that such processes become extinct almost surely.
  
\begin{proof}[Proof of Proposition~\ref{p:PnEvol}]
According to the branching formula \eqref{e:Xevol}
and the fact that if $X_n=i$ then the distribution of the sum is an $i$-fold
convolution of $\hat\pi$, naturally we have
\begin{equation}
\hat\nu_{n+1}(j) = \sum_{i=0}^\infty \hat\nu_n(i) \hat\pi^{*i}(j) 
= \sum_{i=0}^\infty P_n(1,i)P_1(i,j)\,.
\end{equation}
However, one also has another characterization, due to the Markov property.
Namely,
    \begin{equation}
      \hat\nu_{n+1}(j)
	= \sum_{k \geq 0} P_1(1, k) P_n( k, j )
	= \sum_{k \geq 0} \hat \pi(k) \hat\nu_n^{*k}(j) \,,
    \end{equation}
This equation has exactly the same form as~\eqref{e:PnHatEvol}.
Because $\hat\nu_0=\nu_0$, we conclude the proof using induction.
\end{proof}

In light of Proposition~\ref{p:PnEvol}, we identify $\hat\nu_n$ with $\nu_n$ for the remainder of this paper.
While the proof of Proposition~\ref{p:PnEvol} is a simple algebraic calculation, a natural interpretation can be obtained through the work of Kingman~\cite{Kingman82,Kingman82a}.
To see this, fix $N$ large, $n \leq N$ and divide the population at generation $N$ into clans that have a common living ancestor at generation $N - n$.
The clans here play the same role as the coalescing clusters in Section~\ref{s:DCoag}.
Indeed, as $n$ increases we look further back in the ancestry for a common living ancestor, leading to the merger of clans.
Finally, one can directly check that for a given $n$ the distribution of clan sizes is exactly $\hat \nu_n$, and hence~\eqref{e:PnHatSmol}--\eqref{e:IkHat} is expected.

\section{Bernstein transform of the discrete evolution equations}\label{s:BernsteinEvol}

The beautiful, classical theory of Galton-Watson processes is normally developed
in terms of the generating function for the family-size distribution $\hat\pi$, given by
\[
G(z) = \sum_{j=0}^\infty \hat\pi(j)z^j \,.
\]
The generating function of 
the $n^\text{th}$-generation descendant distribution $(\nu_n)$ 
is then given by the $n^\text{th}$ functional iterate of $G$---As is well-known and is easy to derive from \eqref{e:PnHatEvol}, the function
\begin{equation}
G_n(z) = \sum_{j=0}^\infty \nu_n(j)z^j
\end{equation}
satisfies $G_0(z)=z$ and 
\begin{equation}\label{e:Giterate}
G_{n+1}(z) = G(G_n(z)) \,, \quad n\geq0 \,.
\end{equation}
In order to simplify the study of continuum limits and compare with CSBPs, 
however, we find it convenient to recast the formulas of the theory 
using a representation more closely related to Laplace exponents. 

  \begin{definition}
    Given a Galton-Watson process $X$, we define its \emph{Bernstein transform} by
    \begin{equation}\label{e:BtDef}
      \hat \varphi_n(q)
	= \E \paren[\big]{ 1 - e^{-q X_n} }
	= \sum_{j \geq 1} \nu_n(j) (1 - e^{-qj}) \,.
    \end{equation}
  \end{definition}
Of course, this is closely related to the Laplace transform, and also may be expressed
in terms of generating functions, as
\begin{equation}\label{e:VarphiG}
\hat\varphi_n(q) = 1-G_n(e^{-q}) \,.
\end{equation}

The function $\hat\varphi_n$ is a \emph{Bernstein function}~
\cite{SchillingSongEA10}---a nonnegative function
whose derivative is a Laplace transform.
The class of Bernstein functions has a number of convenient 
properties---e.g., it is closed under composition and taking pointwise limits---and 
such functions have proved valuable 
in previous work on coagulation dynamics (see for example~\cite{MenonPego04,MenonPego08,LaurencotRoessel15}).
The first objective of this section is to obtain a convenient expression (equation~\eqref{e:HatVarphi}, below) for the evolution of the Bernstein transform of a Galton-Watson process.

 \begin{proposition}\label{p:BtEvol}
If $\hat \varphi$ is the Bernstein transform of a Galton-Watson process $X$,
 then 
    \begin{equation}\label{e:HatVarphi}
      \boxed{\hat \varphi_{n+1}(q) - \hat \varphi_n(q)
        = - \hat \Psi \paren{ \hat \varphi_n(q) }}
	\qquad\text{    for all $q\geq0$}\,,
    \end{equation}
where 
    \begin{equation}\label{defPsiHat}
\hat\Psi(s) \defeq \sum_{j=0}^\infty(1-s)^j\hat\pi(j) -1 + s  \ =\  G(1-s)-1+s \,,
\end{equation}
in terms of the family-size distribution $\hat\pi$    and its generating function $G$.
  \end{proposition}

\begin{proof}
By \eqref{e:VarphiG} and \eqref{e:Giterate} we have 
\begin{equation}
\hat\varphi_{n+1}(q)-\hat\varphi_n(q) = 
1-G(G_n(e^{-q}))-\hat\varphi_n(q)= -\hat\Psi(\hat\varphi_n(q))\,.
\qedhere
\end{equation}
\end{proof}
  \begin{definition}\label{d:dbran}
    We define the function $\hat \Psi$ in equation~\eqref{defPsiHat} to be the \emph{branching mechanism} of the Galton-Watson process $X$. 
  \end{definition}
  While the generating function has proved extremely useful in many contexts, the notion of branching mechanism as defined above is better suited for our purposes.
As will be seen, it is strongly analogous to the branching mechanism of a CSBP,
it governs convergence of the process in the continuum limit (Proposition~\ref{p:BrMechConv}, below), and it also provides the discrete analog of the Poissonian structure of the rate constants arising in~\cite{BertoinLeGall06} (see~\eqref{e:PsiKRk} below, also revisited later in Section~\ref{s:CSBPintro}).

    We remark that the branching mechanism $\hat \Psi$ is only guaranteed to be defined for $s \in [0, 2]$. It is convex for $s \in [0, 1]$.
In terms of expected family size, it can also be written in the following form. 

\begin{lemma} 
    If the expected family size $\Xi=\sum_{k\geq 0} k\hat\pi(k)$ is finite, then
    the branching mechanism of the Galton-Watson process $X$ satisfies
    \begin{equation}\label{e:PsiHat}
      \hat \Psi(s) = \sum_{j = 2}^\infty
	  \paren[\big]{ \paren{1 - s}^j - 1 + js }
	  \hat \pi(j) + (1-\Xi)s\,.
    \end{equation}
\end{lemma}


  We now compute the rate constants $\hat R_k$ (defined in~\eqref{e:RkHat}) in terms of the branching mechanism.
  \begin{lemma}\label{l:RkHat}
    Let $X$ be a Galton-Watson process with branching mechanism $\hat \Psi$.
    Then for all $k \geq 2$ we have
    \begin{equation}\label{e:PsiKRk}
      \hat R_k(\rho_n) = \frac{(-\rho_n)^k}{k!} \hat \Psi^{(k)}_{\phantom{0}}(\rho_n) \,,
    \end{equation}
    where $\rho_n$ is defined by~\eqref{e:RhoN}
    and $\hat R_k$ is defined by~\eqref{e:RkHat}.
  \end{lemma}
  \begin{proof}
    Termwise differentiation of~\eqref{defPsiHat} gives
    \begin{equation}\label{e:DkPsiHat}
      \frac{(-1)^k}{k!} \hat \Psi^{(k)}_{\phantom{0}}(s)
      =
	  \sum_{j \geq k}
	    \hat \pi(j) \binom{j}{k} (1 - s)^{j-k}\,,\qquad k\geq2 \,,
    \end{equation}
    for $s \in (0,2)$. Because
$\rho_{n} \in (0, 1)$ for all $n$, we may 
    substitute $s = \rho_n$ in~\eqref{e:DkPsiHat} and this gives~\eqref{e:PsiKRk} for $k \geq 2$.
  \end{proof}

  \part{Convergence criteria governing scaling limits of critical Galton-Watson processes.}\label{p:scaling}
  \section{CSBPs and the time continuous Smoluchowski equation.}\label{s:CSBPintro}

  In subsequent sections we will consider scaling limits of Galton-Watson processes,
  focussing on the critical case.
  The limiting processes obtained will be a class of CSBPs, and we use this section to summarize relevant properties of CSBPs.
  We also take this opportunity to indicate similarities between the CSBPs and the discrete notions introduced in Section~\ref{s:BernsteinEvol} that foreshadow results in Section~\ref{s:Scaling}.
  Finally, we describe Bertoin and Le Gall's result~\cite{BertoinLeGall06} relating CSBPs to the time-continuous Smoluchowski equation and compare it to the discrete version introduced in Section~\ref{s:DCoag}.
  \medskip

A CSBP consists of a two-parameter random process $(x, t) \mapsto Z_t(x)\in [0,\infty)$ for $t\geq 0$ and  $x>0$.
For fixed $x$, the process $t\mapsto Z_t(x)$ is a Markov process with initial value $Z_0(x) = x$.
For fixed $t$, the process $x \mapsto Z_t(x)$ is an increasing process with independent and stationary increments.
The right-continuous version of this is a L\'evy process with increasing sample paths.
In particular, the process enjoys the branching property that for all $t \geq 0$, the distribution of $Z_t(x+y)$ is the same as the distribution of the sum of independent copies of $Z_t(x)$ and $Z_t(y)$.

The structure of the process $Z$ has a precise characterization via the Lamperti transform 
(see~\cite{Lamperti67,CaballeroLambertEA09} or \cite[Chap.~12]{Kyprianou14}).
That is, $t \mapsto Z_t(x)$ can be expressed as a subordinated Markov process with parent process $x + \bar X_t$ 
where $\bar X_t$ is a L\'evy process starting from~$0$ with no negative jumps (i.e., $\bar X_t$ is either spectrally positive or a subordinator).
More specifically, $Z_t(x) = x + \bar X_{\Theta(x, t)}$ where the process $t \mapsto \Theta(x, t)$ has non-decreasing sample paths and formally solves $\partial_t \Theta = x + \bar X_\Theta$ with $\Theta(x, 0) = 0$.%
\footnote{%
A derivation of an analog of this for Galton-Watson processes is in Remark~\ref{r:DiscLamperti}, below.
}
In this context, the Laplace exponent of $\bar X_t$, denoted $\Psi$, is called the {\it branching mechanism} for $Z_t(x)$ and has L\'evy-Khintchine representation
\begin{align}\label{e:bm1}
\Psi(q) =
  \frac12 a_0 q^2 -a q -b+ \int_{(0,\infty)} \paren[\big]{ e^{-qx}-1+qx \, \one_{\set{x<1}} } \,d\pi(x) \,,
\end{align}
where $a_0, b \geq 0$, $a \in \R$ and $\pi$ is a positive measure on $(0, \infty)$ satisfying the finiteness condition
\begin{equation*}
  \int_{(0,\infty)} (1 \wedge x^2) \pi (dx) < \infty \,.
\end{equation*}
It is well known~\cites{Grey74,Kyprianou14} that a CSBP remains finite almost surely (or is \emph{conservative})
if and only if
\begin{equation*}
  \int_{(0, 1)} \frac{1}{\abs{\Psi(q)}} \, dq = \infty \,.
\end{equation*}
Clearly, the branching mechanism of a conservative CSBP must have $b = 0$.

A CSBP is subcritical, critical or supercritical if we have $\Psi'(0^+) > 0$, $\Psi'(0^+) = 0$ or $\Psi'(0^+) < 0$ respectively.
Thus for a finite (sub)critical CSBP (by which we mean a critical or subcritical CSBP that remains finite almost surely) we must have $b = 0$ and
\begin{equation}\label{e:PsiPrime1}
  0 \leq \Psi'(0^+) = -a - \int_{[1, \infty)} x \, d\pi(x)\,.
\end{equation}
  Consequently, the branching mechanism of a finite (sub)critical CSBP takes the form 
  \begin{equation}\label{e:Psi}
      \Psi(q)
      = \frac12 a_0 q^2
	+ a_\infty q
	+\int_{(0,\infty)} \frac{ \paren{e^{-qx}-1+qx}}{x} \, d\mu(x)\,,
  \end{equation}
  for some $a_0, a_\infty \geq 0$, together with the finiteness condition 
  \begin{equation*}
    \int_{(0,\infty)} (x\wedge1)\,d\mu(x)<\infty \,.
  \end{equation*}
  Here the measures $\mu$ and $\pi$ are related by $x \, d\pi = d\mu$.
  The reason we introduce the measure~$\mu$ is because with this notation $\Psi'$ is the Bernstein function associated with the L\'evy triple~$(a_0, a_\infty, \mu)$.
  Note $a_\infty = \Psi'(0^+)$ and can be expressed in terms of $a$ using~\eqref{e:PsiPrime1}.

  We remark that \eqref{e:Psi} closely parallels the form of the discrete branching mechanism~$\hat \Psi$ in Definition~\ref{d:dbran}.
  Indeed, for critical Galton-Watson processes $\hat \Psi$ takes the form~\eqref{e:PsiHat},
  which is obtained from~\eqref{e:Psi} by dropping the quadratic term $\frac12 a_0 q^2$,
  setting $a_\infty = 0$ by the criticality condition,
  and using the approximation $e^{-s} \approx 1 -s$.
  We will see later (Proposition~\ref{p:BrMechConv}) that the quadratic term (along with a linear term) reappears in the continuum limit.

  Returning to the CSBP~$Z$, the nature of the Lamperti transform forces the relation
\begin{equation}\label{e:laplace1}
\E \paren[\big]{ e^{-q Z_t(x)} } = e^{-x\vp_t(q)} \,,
\end{equation}
where the spatial Laplace exponent $\vp$ solves the backward equation
  \begin{equation}\label{e:PhiEvol1}
\boxed{   \partial_t \varphi_t(q) = - \Psi(\varphi_t(q))\,,}
    \qquad
    \text{with initial data }
    \quad
    \varphi_0(q) = q \,.
  \end{equation}
This is exactly the continuum analog of the discrete equation~\eqref{e:HatVarphi}, and we will provide a natural and elementary proof of it in Proposition~\ref{p:BTConv}, below.

As the Laplace exponent of a subordinator, $\varphi$ has the L\'evy-Khintchine representation 
\begin{equation}\label{e:khintchine}
    \vp_t(q) = b_t q + \int_{(0,\infty)} (1-e^{-qx}) \, d\bnu_t(x) \,, \quad q\geq 0 \,,
\end{equation}
where $b_t\geq 0$ and $\bnu_t$ is a positive measure satisfying the finiteness condition
\begin{equation*}
  \int_{(0,\infty)} (1\wedge x) \, d\bnu_t(x)<\infty \,.
\end{equation*}
The quantities $b_t$ and $\bnu_t$ are respectively the \emph{drift coefficient} and the \emph{L\'evy jump measure} of the CSBP $Z$.

A striking result of Bertoin and Le Gall~\cite{BertoinLeGall06} shows that the L\'evy jump measure of a \emph{critical} CSBP which becomes extinct almost surely satisfies  a generalized type of \emph{Smoluchowski coagulation equation}.
\ifdraft\sidenote{Sun 08/07 GI: The test functions used in~\cite{BertoinLeGall06} are in $C([0, \infty))$ with $f(0) = 0$. In~\cite{IyerLegerEA15} we use $C([0, \infty])$, and we started writing that here. Should we reconcile? \\ 10/2 RP: \cite{IyerLegerEA15} is correct.\\
  10/5 GI: Agreed.}\fi
Explicitly, Proposition~3 in \cite{BertoinLeGall06} shows that
  \begin{equation}\label{e:Smol}
      \partial_t \ip{f}{\bnu_t} = \sum_{k \geq 2} R_k I_k(\bnu_t,f)\,,
    \qquad\text{for all } f \in C([0, \infty])\,,
  \end{equation}
  where 
  \begin{gather}\label{e:IkDef}
    I_k(\bnu,f) \defeq \int_{(0,\infty)^k} 
      \paren[\Big]{f(x_1+\ldots+x_k)-\sum_{i=1}^k f(x_i) }
      \prod_{i=1}^k \frac{d\bnu(x_i)}{\ip{1}{\bnu}} \,,
    \\
    \label{e:RkPsi}
      R_k \defeq \frac{ (-\rho_t)^k \Psi^{(k)}(\rho_t) }{k!} \,,
      \qquad
       \rho_t \defeq \ip{1}{\bnu_t} \,.
  \end{gather}
  Here
  \begin{equation*}
    \ip{f}{\bnu_t} \defeq \int_{(0,\infty)} f(x)\,d \bnu_t(x),
  \end{equation*}
  is the duality pairing.
  \medskip

Equation~\eqref{e:Smol} has a natural interpretation as a coagulation model introduced by Smoluchowski~\cites{Smoluchowski16,Smoluchowski17} generalized to account for multiple coalescence.
To understand this, we interpret $\set{\bnu_t \st t \geq 0}$ as a family of positive measures on $\R^+=(0,\infty)$, representing the size distribution of clusters.
Namely, $\bnu_t(a, b)$ denotes the {expected} number of clusters at time $t$ that have size in the interval $(a, b)$.

Fix $k \geq 2$ and consider the change in the cluster size distribution due to the simultaneous merger of $k$ clusters.
We assume that the merging clusters are i.i.d.\ with distribution proportional to $\bnu_t$.
For $y \in (0, \infty)$, the merging of smaller clusters into a cluster of size $y$ will result in an increase in the density of clusters of size $y$.
This will happen at a rate proportional to
\begin{equation*}
  \smashoperator{\int_{(0, y)}}
    \frac{ d\bnu_t(x_1) }{\ip{1}{\bnu_t}}
  \smashoperator[r]{\int_{(0, y - x_1)}}
    \frac{ d\bnu_t(x_2) }{\ip{1}{\bnu_t}}
  ~\cdots~
  \smashoperator{\int_{(0, y - \sum_{i=1}^{k-2} x_i )} }
    ~~\frac{ d\bnu_t(x_{k-1}) }{\ip{1}{\bnu_t}}
  \cdot
  \frac{d\bnu_t( y - x_1 - \cdots - x_{k-1} )}{\ip{1}{\bnu_t}}\,.
\end{equation*}
On the other hand, the clusters of size $y$ also combine with larger clusters resulting in a decrease in the density of clusters of size~$y$.
This will happen at a rate proportional to
\begin{equation*}
  k \frac{d\bnu_t(y)}{\ip{1}{\bnu_t}} \,.
\end{equation*}
Thus, for any test function $f \in C([0, \infty])$, the rate at which the simultaneous merger of $k$ clusters affects the moment $\ip{f}{\bnu_t}$ is proportional to the difference of the above two terms integrated against $f$.
Changing variables we see that this is exactly $I_k(f , \bnu_t)$ and hence the rate at which the simultaneous merger of $k$ clusters affects $\ip{f}{\bnu_t}$ is proportional to $I_k(f, \bnu_t)$.
Summing over $k$ and multiplying by proportionality constants explains how~\eqref{e:Smol} models coalescence.

In general the rate constants $R_k$ appearing in~\eqref{e:Smol} can be chosen arbitrarily.
In the context of CSBPs, the $R_k$'s have a special Poissonian structure given by~\eqref{e:RkPsi}.
One of the main motivations of our exposition is to provide a clear account of
the meaning of the measure $\bnu_t$ in this context, the precise way it arises in the continuum limit, and how it comes to be governed by coagulation dynamics with the indicated rates.

Precisely, we  will show that for a finite (sub)critical CSBP, the L\'evy jump measures $\bnu_t$ arise as the scaling limit of the $n^\text{th}$ generation descendant distributions of rescaled critical Galton-Watson processes~$\nu_n$.
To briefly explain the main idea, recall Proposition~\ref{p:PnEvol} shows that $\nu_n$ satisfies~\eqref{e:PnHatSmol}.
This is of course simply a discrete version of~\eqref{e:Smol}.
  Indeed, given a sequence $\set{f(j)}_{j \in \N}$, we multiply~\eqref{e:PnHatSmol} by $f(j)$ and sum over $j$ to obtain
  \begin{equation}\label{e:DiscSmolNu}
    \sum_{j > 0} f(j) (\nu_{n+1}(j) - \nu_n(j))
      = \sum_{k \geq 2} \hat R_k(\rho_n) \hat I_k(\nu_n, f)\,,
  \end{equation}
  where 
  \begin{equation}\label{e:DiscIkNu}
    \hat I_k( \nu_n, f )
      \defeq \sum_{i_1, \dots, i_k > 0}
      \brak[\Big]{
	f\paren[\Big]{\sum_{l = 1}^k i_l} - \sum_{l = 1}^k f(i_l)
      }
      \prod_{l = 1}^k \frac{\nu_n(i_l)}{\rho_n}\,.
  \end{equation}
  Moreover, Lemma~\ref{l:RkHat} shows that the rate constants are obtained from the discrete branching mechanism $\hat \Psi$ (Definition~\ref{d:dbran}) in exactly the same manner as~\eqref{e:RkPsi}.

  Thus, after rescaling the Galton-Watson processes correctly, it is only natural to expect convergence of the rescaled branching mechanisms of the Galton-Watson process to the branching mechanism of the CSBP, and convergence of the rescaled descendant distributions $\nu_n$ to the L\'evy jump measure $\nu_t$.
  Moreover, the L\'evy jump measure $\nu_t$ should satisfy~\eqref{e:Smol} if the limit is critical, and~\eqref{e:Smol} with an additional damping term (analogous to~\eqref{e:PnHatSmolSubC}) if the limit is subcritical.
  We prove this in Proposition~\ref{p:BrMechConv} and Corollaries~\ref{clyLevyTripConv} and~\ref{c:CSBPSmol} below.

%
%
%
%

  \section{Scaling limits of critical Galton-Watson processes.}\label{s:Scaling}

In this section we study scaling limits of critical Galton-Watson processes using the discrete coagulation dynamics developed above.

We establish necessary and sufficient criteria for convergence of the discrete branching mechanisms (Proposition~\ref{p:BrMechConv}), convergence of the Bernstein transforms (Proposition~\ref{p:BTConv}) and of the rescaled Galton-Watson processes themselves (Proposition~\ref{p:GWConv}), in terms of a type of weak convergence of the reproduction laws alone.  
Moreover, we show (Corollaries~\ref{clyLevyTripConv} and \ref{c:CSBPSmol}) that the L\'evy jump measure of the limiting CSBP satisfies a generalized (damped) Smoluchowski equation.
The precise notion of convergence is naturally associated with continuity theorems for Bernstein transforms, which we develop in Appendix~\ref{s:CT} due to their independent interest. 

  \subsection{Rescaled time-discrete dynamics.}
  We begin by rescaling the coagulation model~\eqref{e:RhoN}--\eqref{e:IkHat} (where $\hat R_k$ is defined by~\eqref{e:RkHat}).
  Let   $h > 0$ be a grid size, and  $\tau > 0$ be a time step.
  We rescale the variables so that cluster sizes are integer multiples of $h$, and the merger of clusters happens on intervals of time $\tau$.
  Further, in order to facilitate passing to the limit as $h,\,\tau\to0$, 
  we associate measures supported on the grid $h \N$ to the rescaled size distributions.
  Explicitly, we define
  \begin{equation}\label{e:DefTiPi}
    \suph \nu_n = \frac{1}{h} \sum_{j \geq 1} {\nu_n}(j) \delta_{jh}\,.
        \qquad 
        \subht \pi = \frac{1}{\tau h} \sum_{j \geq 1} \hat \pi(j) \delta_{jh}\,.
  \end{equation}
  Here $\delta_x$ denotes the Dirac measure centered at $x$.

  In the context of Galton-Watson processes, the above corresponds to scaling population by a factor of~$h$ and reproducing at times which are integer multiples of~$\tau$.
  That is, the rescaled process $Y$ is given by
  \begin{equation*}
    Y_{n \tau}(j h) = h X_n(j) \,.
  \end{equation*}
  We will, however, postpone the discussion of rescaled Galton-Watson processes to Section~\ref{s:GW}, and instead study the rescaled size distributions first.

Associated with \eqref{e:DefTiPi} we denote the Bernstein transform of $\suph\nu_n$ by
    \begin{equation}\label{e:TiPhinDef}
      \suph \varphi_n(q) = \int_{\R^+} \paren{1 - e^{-qx}} \, d\suph \nu_n(x)
      = \frac1h \hat\varphi_n(hq) \,.
    \end{equation}
We assume $\hat\pi$ is critical, and define a rescaled branching mechanism
by  
  \begin{equation}\label{e:DefTiPsi}
    \subht \Psi(q) \defeq \int_{\R^+} \paren[\big]{ (1 - hq)^{x/h} - 1 + q x } \, d\subht \pi(x)
    = \frac{\hat \Psi(h q)}{\tau h} \,.
  \end{equation}
  Note that $\subht \Psi$ is only guaranteed to be defined on the interval $[0, 2/h]$.
  It is increasing and convex on $[0, 1/h]$, and
  for this reason we will subsequently ensure $0 \leq q \leq 1/h$ whenever we 
  use~$\subht \Psi(q)$.

%
%
%

The evolution equations and their Bernstein transforms now take the following form.
  \begin{lemma}
    Using the rescaled variables in~\eqref{e:DefTiPi}, 
    equation~\eqref{e:DiscSmolNu} becomes
    \begin{equation}
      \frac{\ip{f}{\suph \nu_{n+1}}	  - \ip{f}{\suph \nu_{n}} }{\tau}
      = \sum_{k \geq 2} \supht R_k(\suph\rho_n) I_k( \suph \nu_n, f )
    \end{equation}
    for any bounded $f\in C(\R_+)$.
    Here $I_k$ is defined by~\eqref{e:IkDef}, and
   \begin{equation}
    \supht R_k (\suph\rho_n)
      \defeq \frac{(-\suph \rho_n)^k}{k!} \subht \Psi^{(k)}(\suph \rho_n ) \,,
  \end{equation}
  where $\suph \rho_n = \ip{1}{\suph \nu_n}$ represents the rescaled total number at time $n \tau$.

Further, the Bernstein transform of $\suph \nu_n$ satsifies
    \begin{equation}\label{e:PhiDiff}
\boxed{      \frac{ \suph \varphi_{n+1}(q) - \suph \varphi_n(q) }{\tau}
      = - \subht \Psi( \suph \varphi_n(q) ) }
      \qquad\text{for all } n\geq0,\, q\geq0 \,.
    \end{equation}
  \end{lemma}
  \begin{proof}
    The proof is a direct computation using~\eqref{e:DiscSmolNu}, Lemma~\ref{l:RkHat}, and~\eqref{e:HatVarphi}.
  \end{proof}


  \subsection{Convergence of critical branching mechanisms.}

The first step in studying continuum limits of $\suph \nu_n$ is to study convergence of the branching mechanisms.  We obtain necessary and sufficient conditions for such convergence,
in terms of a criterion that is closely tied to the continuity theorems relating Bernstein transforms and L\'evy triples, which we develop in the Appendix. 

For greater generality, we will study sequential limits where we also allow the measure $\hat \pi$ to vary. Let $(\hat\pi_k)$ be a sequence of probability measures on $\nplus$
which satisfy the criticality condition~\eqref{e:Criticality}, 
and let $(h_k)$, $(\tau_k)$ be positive sequences converging to $0$. 
We introduce (rescaled) discrete branching mechanisms 
as in \eqref{defPsiHat} and \eqref{e:DefTiPsi} 
with $\hat\pi_k$ replacing $\hat\pi$, by defining 
\[
\hat\Psi_k(q) \defeq  \sum_{j = 2}^\infty
	  \paren[\big]{ \paren{1 - q}^j - 1 + jq }
	  \hat \pi_k(j) \,, \qquad 
\scalek\Psi(q)\defeq\frac{\hat\Psi_k(h_kq)}{\tau_k h_k}\,.
\]

We next associate to each family size distribution 
$\hat\pi_k$  a (L\'evy) measure $\hat\mu_k$ on $\R^+$ given by
\begin{equation}\label{defMuHat}
d\hat\mu_k(x) \defeq \sum_{j\geq2}(j-1)\hat\pi_k(j)\,d\delta_{j}(x) \,.
\end{equation}
This measure is rescaled according to the relation
\begin{equation} \label{defMuht}
d\scalek\mu(x) 
\defeq \frac1{\tau_k} d\hat\mu_k\left(\frac x{h_k}\right) =
\frac1{\tau_k}\sum_{j\geq2} (j-1)\hat\pi_k(j)\, d\delta_{jh_k}(x)\,.
\end{equation}
The next definition singles out a particular sense of convergence of these measures that will be important throughout the rest of this paper.
This notion relates to convergence of L\'evy triples and is revisited in Appendix~\ref{s:CT} (see Remark~\ref{r:Lconv}).
\begin{definition}\label{d:Lconv}
Given some finite measure $\kappa$ on $[0,\infty]$, we say 
the sequence $(\scalek\mu)$ \emph{L\'evy-converges} to $\kappa$ provided
\begin{equation}\label{e:MuLim}
    (x\wedge 1)\,d\scalek\mu(x) \to d\kappa(x) 
\quad\text{weak-$\star$ on $[0,\infty]$.}
\end{equation}
\end{definition}
Recall a sequence of finite measures $(\kappa_k)$ converges to $\kappa$ weak-$\star$ on $[0, \infty]$ if for every test function $g \in C([0, \infty])$ we have $\ip{g}{\kappa_k} \to \ip{g}{\kappa}$.   
We require test functions to be continuous at $\infty$ in order to capture any atom at $\infty$.

\begin{proposition}\label{p:BrMechConv}
Given a sequence $(\hat\pi_k)$ satisfying the criticality condition \eqref{e:Criticality}, and positive sequences $(h_k)$, $(\tau_k)$
converging to zero, let $(\scalek\mu)$, $(\scalek\Psi)$ be as above.

\begin{asparaenum}[(i)]
  \item Suppose that $(\scalek\mu)$ L\'evy-converges to 
some finite measure $\kappa$ on $[0,\infty]$ as $k\to\infty$.
Then for each $q\in [0,\infty)$, as $k\to\infty$  we have
    \begin{equation}\label{e:PsiLim}
\scalek\Psi(q)\to	\Psi(q)
	\defeq \frac12 \alpha_0 q^2+\alpha_\infty q
	+\int_0^\infty \frac{e^{-qx}-1+qx}x \, d\mu(x)\,,
    \end{equation}
where $(\alpha_0,\alpha_\infty,\mu)$ is the L\'evy triple associated with $\kappa$ by the relation
\begin{equation}\label{e:KappaMu}
d\kappa(x) 
	= \alpha_0 \, d\delta_0 + \alpha_\infty \, d\delta_\infty+ (x\wedge1)\, d\mu(x)\,.
 \end{equation}
 Moreover, each derivative $\scalek\Psi^{(m)}$ converges to $\Psi^{(m)}$, 
 locally uniformly in $(0,\infty)$ for each $m\in\N$.  
 
 \item Conversely, suppose $\Psi(q)=\lim_{k\to\infty}\scalek\Psi(q)$ exists
 for each $q\in[0,\infty)$. Then $(\scalek\mu)$ L\'evy-converges to
 some finite measure $\kappa$ on $[0,\infty]$, and $\Psi$ is given by 
 \eqref{e:PsiLim}--\eqref{e:KappaMu}.
\end{asparaenum}
\end{proposition}

  \begin{remark*} 
    The L\'evy-convergence requirement~\eqref{e:MuLim} corresponds exactly to convergence of L\'evy triples in a natural topology associated to subordinators.
    Explicitly, criterion~\eqref{e:MuLim} is equivalent to convergence of the L\'evy triples $(0,0,\scalek\mu)$ to the L\'evy triple $(\alpha_0,\alpha_\infty,\mu)$, as described in Appendix~\ref{s:CT}.
  \end{remark*}

\begin{remark*}  
  Recall that the expression for $\Psi$ in \eqref{e:PsiLim} is the general form of a branching mechanism for a finite (sub)critical CSBP (see~\cite[Ch.~12]{Kyprianou14} or Section~\ref{s:CSBPintro}).
  We show in Section~\ref{s:Univ}, below, that every such branching mechanism does arise as a sequential limit from discrete branching mechanisms of \emph{critical} Galton-Watson processes.
  A heuristic explanation as to why critical branching mechanisms might yield a subcritical branching mechanism in the limit is discussed in Remark~\ref{r:subcLim}, below.
\end{remark*}

\begin{proof}[Proof of Proposition~\ref{p:BrMechConv}.]
We begin by proving (i).

\begin{asparaenum}[1.]
\item Define 
\begin{equation}\label{defKappak}
d\kappa_k = (x\wedge 1)\,d\scalek\mu\,.
\end{equation}
Then $\kappa_k\to\kappa$ weak-$\star$ on $[0, \infty]$ by \eqref{e:MuLim}.
Next fix $q\geq0$, and compute  
    \begin{align}
      \label{e:DPsi}
      \Psi'(q) &= \alpha_0 q +\alpha_\infty+ \int_0^\infty (1-e^{-q x}) \, d \mu(x)
	= \ip{f_0}{\kappa}\,,   
	     \\    \label{e:DPsik}
      \scalek \Psi'(q) &= \int_{[2h_k,\infty)}
	\frac{1-(1-h_k q)^{(x-h_k)/h_k}}{x-h_k}  {x}\,d\scalek \mu(x)
	= \ip{\subhk f}{\kappa_k} \,,
    \end{align}
    where
    \begin{align*}
       f_0(x)	= \frac{1-e^{-q x}}x
	\left( \frac x{x\wedge1} \right),      
      \quad
       f_h(x)	= { \frac{1-(1-h q)^{(x-h)/h}}{x-h} }
	  \paren[\Big]{ \frac{x}{x\wedge 1} }\,.
    \end{align*}
The second equality in~\eqref{e:DPsik} follows because $\kappa_k$ is supported on $[2h_k, \infty)$.

Note $ \Psi(0) = \subhk\Psi(0) = 0$, and $\Psi'$ and $\subhk\Psi'$ are positive and increasing.
    Thus, the desired conclusion in \eqref{e:PsiLim}  for fixed $q$ will follow, provided we show
    \begin{equation}\label{e:KappaKConv}
      \ip{\subhk f}{ \kappa_k} = \ip{\subhk f-f_0}{ \kappa_k}
      +\ip{f_0}{ \kappa_k}\to\ip{f_0}{\kappa}\,.
    \end{equation}
    Clearly $\ip{f_0}{\kappa_k}\to\ip{f_0}{\kappa}$ because $f_0\in C([0,\infty])$.
    We claim
    \begin{equation}\label{e:FkMinusFh}
      \lim_{h \to 0} \sup_{x > h} \ \abs[\big]{ f_h(x)-f_0(x) } = 0 \,.
    \end{equation}
    Because $\ip{1}{\kappa_k}\to\ip{1}{\kappa}$, this immediately implies~\eqref{e:KappaKConv}.

  \item
    To finish the proof of \eqref{e:PsiLim}, it only remains to prove~\eqref{e:FkMinusFh}.
    For this, we claim that, provided $\max(h, q h)<\frac12$,
    \begin{equation}\label{e:FkMinusF0}
      \abs{ f_h(x)-f_0(x) }
      \leq
	\begin{dcases}
	  5 e^{-q x}+ \frac{2 h}{x} & \text{for } x>2 \,,\\
	  2 (1 \varmax x) h q^2  & \text{for } x>h \,.
	\end{dcases}
    \end{equation}
    Indeed, this estimate in the case $x > 2$ follows immediately from the bounds
    \begin{gather*}
      0 < \frac{x}{x-h}-1
	= \frac{h}{x-h} <\frac{2h}x,
      \\
      \frac{x}{x-h} (1- q h)^{(x-h)/h}
	\leq 2\frac{ (e^{-q h})^{x/h}} {1- q h}
	\leq 4e^{-q x}\,.
    \end{gather*}
In the case $x > h$, observe that because $z\mapsto e^{z}$ is contractive for $z<0$,
we have
\begin{align*}
      \frac{ \abs[\big]{(1-q h)^{(x-h)/h} -e^{-(x-h)q}} }{x-h}
	&\leq \frac{1}{h} \abs{\ln(1-q h)+q h}
      = \frac{1}{h} \int_0^{q h} \frac z{1-z}\,dz \leq h q^2\,,
\end{align*}
because $qh<\frac12$. Moreover, because $z\mapsto (1-e^{-z})/z=\int_0^1 e^{-rz}\,dr$
is a decreasing contraction for $z>0$, 
\[
    0 < \frac{1-e^{-(x-h)q}}{x-h} - \frac{1-e^{-xq}}x 
    \leq \abs[\big]{ (x - h) q - xq } q
    = h q^2 \,.
\] 
By adding these last two bounds and using $x/(x\wedge1)=1\vee x$, we infer~\eqref{e:FkMinusF0}.

    Using the first estimate in~\eqref{e:FkMinusF0} for $x>h^{-1/2}$, and the
    second estimate in~\eqref{e:FkMinusF0} for $x \leq h^{-1/2}$, we obtain~\eqref{e:FkMinusFh}.
    This finishes the proof of \eqref{e:PsiLim}.


\item
  To prove the statement regarding local uniform convergence, 
let $\Omega$ be an open set with compact closure 
in the right half plane $\set{q\in\C\mid\re q>0}$, and note that  
$|1-hq|<1$ for all $q\in\Omega$, for sufficiently small $h>0$.
When $x=mh$ for an integer $m\geq2$, the function $q\mapsto f_h(x)$ is analytic 
and is clearly bounded on $\Omega$ uniformly for $h$ small and for $x\geq 1$.
A uniform bound holds for $x\leq 1$ as well due to the fact that 
in this case $ \partial_q f_h(x) = (1-hq)^{m-2}$,
whence $|\partial_q f_h(x)|\leq 1$.

Now, because of \eqref{e:DPsik}, the functions $(\scalek\Psi')_{k\geq N}$ are 
analytic and uniformly bounded on $\Omega$.
By Montel's theorem, this sequence converges uniformly on $\Omega$, 
and by Cauchy's integral formula, all derivatives converge locally uniformly.
This finishes the proof of part (i).
\end{asparaenum}

For part (ii), assume $\Psi(q)=\lim_{k\to\infty}\scalek\Psi(q)$ exists for each $q\in [0,\infty)$.
 First, considering some $q$ fixed,  note that for every $x>h$, 
 \[
 f_h(x) \geq \frac{1-e^{-q(x-h)}}{x-h} (1\vee x) \geq \frac{1-e^{-qx}}x(1\vee x) =f_0(x) \,.
 \]
Therefore by~\eqref{e:DPsik} we have that whenever $k$ is so large that $2q h_k <1$,
    \begin{equation}\label{e:KappaKBd}
(\inf f_0) \ip{1}{ \kappa_k}
      \leq \ip{\subhk f}{ \kappa_k} 
      = \scalek \Psi'(q)
      \leq \frac{\scalek \Psi(2q)}{q} \,.
    \end{equation}
The last inequality holds because $\scalek \Psi$ is convex and positive on $(q,2q)$.
    Because $\inf f_0>0$ and $(\scalek\Psi(2q))$ is bounded, 
    it follows that $\sup_k \ip{1}{ \kappa_k} < \infty$, and hence $\set{ \kappa_k}$ is weak-$\star$ pre-compact.

Thus, any subsequence of $(\kappa_k)$ has a further subsequence that converges weak-$\star$
on $[0,\infty]$. Let $\kappa$ denote any such limit. By the proof of part (i) above, 
we infer that for any $q>0$, as $k\to\infty$ along the appropriate subsequence,
\[
\scalek\Psi'(q) = \ip{\subhk f}{\kappa_k} \to \Phi(q) \defeq \ip{f_0}{\kappa} \,.
\]
By dominated convergence we deduce $\Psi(q) = \lim \int_0^q \scalek\Psi'(r)\,dr = \int_0^q\Phi(r)\,dr$,
whence we conclude $\Psi$ is $C^1$ and $\Phi=\Psi'$.  

Let $(\alpha_0,\alpha_\infty,\mu)$ be the L\'evy triple associated with $\kappa$ by \eqref{e:KappaMu}.
Then we see that \eqref{e:DPsi} holds, and $\Psi'$ is the Bernstein transform of this triple.
Because the Bernstein transform is bijective (see Theorem~\ref{t:BernsteinRep} of the Appendix),
it follows that $\Psi$ determines $\kappa$ uniquely. By consequence, the whole sequence 
$(\kappa_k)$ must converge, meaning that \eqref{e:MuLim} holds. This finishes the proof.
%
  \end{proof}

  \subsection{Convergence of Bernstein transforms}

Next we study the convergence of the rescaled size distributions $\nu^h_n$ for survivors,
by the simple expedient of studying convergence of their Bernstein transforms $\varphi^h_n$, regarding \eqref{e:PhiDiff} as a forward Euler difference approximation to the ODE \eqref{e:PhiEvol1}. 




\begin{proposition}\label{p:BTConv}
Make the same assumptions as in Proposition~\ref{p:BrMechConv}.

\begin{asparaenum}[(i)]
  \item
    Suppose that the L\'evy-convergence in \eqref{e:MuLim} holds
    for some finite measure $\kappa$ on $[0,\infty]$. 
    Then, as $k\to\infty$ and whenever $n\tau_k\to t$, 
    we have 
\begin{equation}\label{e:PhihkConv}
\suphk\varphi_n(q)\to\varphi(q,t)
\quad\text{ for every } q , t \in [0, \infty) \,,
\end{equation}
    where $\varphi$
is the unique solution of
    \begin{equation}\label{e:PhiEvol}
      \partial_t \varphi(q, t) = - \Psi ( \varphi(q, t))\,,
      \qquad
      \varphi(q, 0) = q \,,
    \end{equation}
with $\Psi$ given by \eqref{e:PsiLim}.
Moreover, $\varphi(\cdot, t)$ is Bernstein for each $t \geq 0$, and has the form
\begin{equation}\label{e:PhiBern1}
\varphi(q,t) = \beta_0(t) q +\beta_\infty+ \int_{(0,\infty)} (1-e^{-qx})\,d\nu_t(x)\,,
\end{equation}
where $\beta_\infty=0$, $\beta_0(t)\geq0$ and $\int_{(0,\infty)} (x\wedge1)\,d\nu_t(x)<\infty$. 

  \item
    Conversely, suppose \eqref{e:PhihkConv} holds, and $\varphi(q_0, t_0) > 0$ for some $q_0,t_0 > 0$.
 Then \eqref{e:MuLim} holds  for some finite measure $\kappa$ on $[0,\infty]$.
\end{asparaenum}
  \end{proposition}


  \begin{proof}
   We prove (i) by providing a rather straightforward
    proof that the forward Euler difference scheme in \eqref{e:PhiDiff} converges
    if \eqref{e:PsiLim} holds.
    Fix $q \in \R^+$ and $T > 0$.
    From equation~\eqref{e:PhiEvol} we see
    \begin{equation*}
      \varphi( q, t)
	= q - \int_0^{t} \Psi ( \varphi( q, s ) )\, ds.
    \end{equation*}
Because $s\mapsto\Psi( \varphi( q, s )) $ is positive decreasing,
$\varphi(q,s)\leq q$, 
whenever $n\tau_k \leq T$ we have
    \begin{equation*}
	\varphi( q, n\tau_k )
	= q - \tau_k \sum_{m=0}^{n - 1} \Psi ( \varphi(q, m \tau_k)) + R_1' \,,
    \end{equation*}
where 
    \begin{equation}
0\leq R_1'\leq \tau_k \sum_{m=0}^{n - 1} \paren[\Big]
{\Psi ( \varphi(q, m \tau_k) )- \Psi ( \varphi(q, (m+1) \tau_k))}
\leq \tau_k\Psi(q)\,.
    \end{equation}
Let $\varphi_k(q,n\tau_k) \defeq \suphk\varphi_n(q)$ and $\Psi_k \defeq \subhtk\Psi$.  Summing~\eqref{e:PhiDiff}, and noting $\nu_0^h =\delta_h/h$, we have%
    \begin{align} \label{e:PhikSum}
      \varphi_k( q, n\tau_k )
	&= \frac{1 - e^{-qh_k}}{h_k} - 
	\tau_k \sum_{m = 0}^{n-1} \Psi_k( \varphi_k(q, m \tau_k) )
	\\&
	= q - \tau_k \sum_{m = 0}^{n-1} \Psi( \varphi_k(q, m \tau_k) ) + R_2'\,,
  \nonumber
    \end{align}
    where, because $n\tau_k\leq T$,
    \begin{equation}
      \abs{R_2'} \leq M_k \defeq {q^2 h_k} + T\sup_{[0, q]} \abs{ \Psi_k - \Psi}\,.
    \end{equation}
Consequently, because $\Psi'$ is positive and increasing on $[0,q]$,
    \begin{align*}
      \abs{(\varphi - \varphi_k)( q, n\tau_k)}
	&\leq  \abs{ R_1'}+ \abs{R_2'} +
	\tau_k
	  \sum_{m=0}^{n-1}
	  \abs[\big]{ \Psi( \varphi(q, m \tau_k)) - \Psi(\varphi_k( q, m \tau_k )) }
	\\
	&\leq \tau_k\Psi(q)+M_k
	  +\tau_k \Psi'(q)
	  \sum_{m=0}^{n-1}
	  \abs[\big]{ (\varphi - \varphi_k)( q, m \tau_k ) }\,.
    \end{align*}
Hence, by the discrete Gronwall inequality,
\begin{equation*}
      \abs{(\varphi - \varphi_k)( q, n\tau_k)}  
	\leq \paren[\big]{ \tau_k\Psi(q)+M_k}  
	  \paren[\big]{1 + \tau_k \Psi'(q)}^n
	\leq \paren[\big]{ \tau_k\Psi(q)+M_k} 
	  e^{\Psi'(q) n \tau_k}.
\end{equation*}
Proposition~\ref{p:BrMechConv} guarantees $M_k\to 0$ as $k \to \infty$.
Because $\varphi(q,t)$ has bounded derivative $|\partial_t\varphi|\leq\Psi (q)$,
we may infer that whenever $n\tau_k\to t\in [0,T)$,
\begin{equation*}
  \abs{\varphi_k(q,n\tau_k)-\varphi(q,t)} \to 0
    \quad\text{as } k\to\infty \,.
\end{equation*}
This finishes the proof of convergence of $(\varphi_k)$.
Because $\varphi_k(\cdot,n\tau_k)$ is Bernstein, and any pointwise limit of Bernstein
functions is Bernstein \cite[Cor.~3.7]{SchillingSongEA10}, the limit $\varphi(\cdot,t)$ is Bernstein and has the form in \eqref{e:PhiBern1} for some L\'evy triple 
$(\beta_0,\beta_\infty,\nu_t)$. We must have $\beta_\infty=0$, however, because
$\varphi(q,t)\leq q$ for all $q,t\geq0$.
\smallskip

 For the proof of (ii), choose $q, t>0$ such that $\varphi(q, t) > 0$.
 From~\eqref{e:PhikSum} we find that $\varphi_k(q, m \tau_k) \geq \varphi_k(q, n \tau_k)$ for $m \leq n$.
 Consequently, for $k$ sufficiently large, with $n=\floor{t/\tau_k}$ we have $n\tau_k>t/2$
 and $\varphi_k(q,n \tau_k)>\varphi(q,t)/2$, and we find from \eqref{e:PhikSum} again that
 \[
 \frac{t}2 \Psi_k\left(\frac{\varphi(q,t)}2\right) \leq
 \tau_k \sum_{m = 0}^{n-1} \Psi_k( \varphi_k(q, m \tau_k) )
 \leq \frac{1 - e^{-qh_k}}{h_k} \leq q \,.
 \]
%
Hence $\set{  \Psi_k( \varphi(q, t) / 2 ) }$ is bounded.
    Following the proof of Proposition~\ref{p:BrMechConv} (specifically, using~\eqref{e:KappaKBd}) we see that the set of measures $\set{\kappa_k}$ is weak-$\star$ pre-compact.
    (Here~$\kappa_k$ is  defined by~\eqref{defKappak}.)

Thus, any subsequence of $(\kappa_k)$ has a further subsequence that converges weak-$\star$ on $[0,\infty]$. Let $\kappa$ denote any such limit. 
Then, by part (i) above, we know \eqref{e:PhiEvol} holds with $\Psi$ given
 by taking the limit in \eqref{e:PsiLim}  along the appropriate subsequence, 
 as in the proof of part (ii) of Proposition~\ref{p:BrMechConv}. 
 Then $\varphi$ is $C^1$, and $\Psi$ is determined by $\varphi$ by evaluating
 \eqref{e:PhiEvol} at $t=0$. As in the proof of Proposition~\ref{p:BrMechConv},
 it follows $\kappa$ is determined by $\varphi$, hence the whole
 sequence $(\kappa_k)$ converges. This means \eqref{e:MuLim} holds,
 finishing the proof.
 \end{proof}


  Because convergence of the Bernstein transforms is equivalent to 
  convergence of corresponding L\'evy triples and 
  weak-$\star$ convergence of corresponding measures on $[0,\infty]$ by the continuity theorem~\ref{t:Cont}, we obtain the following corollary.

\begin{corollary}\label{clyLevyTripConv}
Under the same assumptions as Proposition~\ref{p:BTConv} part (i), 
we have that 
for every $t \geq 0$,  as $k\to\infty$ and whenever $nh_k\to t$,
    \begin{equation}\label{e:NukLim}
    (x\wedge1)\,d\suphk\nu_{n}(x) \to \beta_0(t)\,d\delta_0(x)+ (x\wedge1)d\nu_t(x)  \quad\text{ weak-$\star$ on $[0, \infty]$.}
    \end{equation}
    Conversely, if this convergence holds, then~\eqref{e:MuLim} holds for some finite measure $\kappa$ on $[0,\infty]$. 
    
Further, we have $\beta_0(t) =0$ for all $t>0$ if and only if in 
\eqref{e:PsiLim} we have $\Psi'(\infty)=\infty$, or equivalently
\begin{equation} \label{e:BetaZero}
 a_0 >0 \quad\mbox{or}\quad \int_{(0,\infty)} d\mu(x) = \infty\,.
\end{equation}
In case $\Psi'(\infty)<\infty$ then $\beta_0(t)=\exp(-\Psi'(\infty)t)$.

Moreover, we have $\varphi(\infty,t)=\ip{1}{\nu_t}<\infty$ if and only if Grey's condition holds:
\begin{equation}\label{e:Grey}
  \int_{[1,\infty)} \frac{du}{\Psi(u)} < \infty\,.
\end{equation}
In this case, the limiting family of finite measures $(\nu_t)_{t\geq0}$ is a weak solution of the generalized damped Smoluchowski equation
\begin{equation}\label{e:SmolSubC}
  \partial_t \ip{f}{\bnu_t} = - \Psi'(0^+) \ip{f}{\bnu_t} + \sum_{k \geq 2} R_k I_k(\bnu_t,f) \,,
\end{equation}
for all test functions $f \in C([0, \infty])$.
Here $R_k$ and $I_k$ are as in~\eqref{e:IkDef} and~\eqref{e:RkPsi} respectively.
\end{corollary}
\begin{remark*}
  In the critical case $\Psi'(0^+) = 0$ and~\eqref{e:SmolSubC} is precisely the generalized Smoluchowski equation~\eqref{e:Smol} encountered before.
\end{remark*}

  \begin{remark*}
  The convergence property \eqref{e:NukLim} explains precisely how the L\'evy jump measure of a (sub)critical CSBP $X$ arises as a limit of scaled $n$-th generation descendant distributions.
  This was one of our main motivations for developing this treatment of continuum limits of Galton-Watson processes.
  \end{remark*}

\begin{remark*}
For the case when Grey's condition \eqref{e:Grey} does not hold, Bertoin and Le Gall 
have proposed an interesting generalization of Smoluchowski's equation in terms of the
sum of locations of atoms of Poisson random measures; 
see~\cite[Eq.~(26)]{BertoinLeGall06}.
\end{remark*}

  \begin{proof}
    \begin{inparaenum}[1.~]
    \item The weak convergence in \eqref{e:NukLim} and the converse follow immediately from Proposition~\ref{p:BrMechConv} and the continuity theorem for Bernstein transforms (Theorem~\ref{t:Cont}).

    \item
      Next, note that by \eqref{e:PhiBern1} and because $(1-e^{-qx})/q\leq x\wedge(1/q)\to0$
as $q\to\infty$,
\begin{equation}\label{e:Beta0}
\beta_0(t) = \lim_{q\to\infty} \frac{\varphi(q,t)}{q}\,.
\end{equation}
We now claim $\beta_0(t)=0$ for all $t>0$ if and only if $\lim_{q\to\infty}\Psi'(q)=\infty$.
Indeed, if $\beta_0(t)\geq\eps>0$ for some $t>0$, then because $\varphi(q,t)$ is 
concave in $q$ and decreasing in $t$, necessarily $\varphi(q,t)\geq \eps q$ and 
\[
q \geq \int_0^t\Psi(\varphi(q,s))\,ds \geq t \Psi(\eps q)
\]
for all $q>0$. Because $\Psi$ is convex it follows $\Psi'(q)\leq 1/(\eps t)$ for all $q$.
Conversely, if $\Psi'(q)\leq M$ for all $q$ then $\Psi(q)\leq Mq$ and 
$\varphi(q,t)\geq qe^{-Mt}$ by \eqref{e:PhiEvol}.


Now, the condition $\lim_{q\to\infty}\Psi'(q)=\infty$ is easily seen to be equivalent 
to \eqref{e:BetaZero}. This establishes both as necessary and sufficient to
have $\beta_0(t)\equiv0$.
In the case $\Psi'(\infty)<\infty$, then due to \eqref{e:PhiEvol} and \eqref{e:Beta0}
we find
\[
\beta_0(t)-\beta_0(s) = -\lim_{q\to\infty} \frac1q \int_s^t \Psi(\varphi(q,r))\,dr =
 -\int_s^t \Psi'(\infty)\beta_0(r)\,dr\,,
\]
whence $\beta(t)=\exp(-\Psi'(\infty)t)$ since $\beta_0(0)=1$.

\item That Grey's condition is necessary and sufficient for $\ip 1{\nu_t}<\infty$ 
is well-known (see for example~\cite[proof of Theorem 12.5]{Kyprianou14})
and is easily deduced by separating variables in \eqref{e:PhiEvol} and integrating.

\item Finally, we show that $(\nu_t)$ satisfies~\eqref{e:SmolSubC}.
  In the case that~$\Psi$ is critical (i.e.\ $\Psi'(0^+) = 0$) a proof can be found in \cite[Proposition~3]{BertoinLeGall06} or \cite[Theorem~3.2]{IyerLegerEA15}.
  In the general case the proof closely resembles~\cite[Theorem~3.2]{IyerLegerEA15} and we sketch the details here.

  The main idea is to establish~\eqref{e:SmolSubC} for test functions of the form
  $
    f_q(x) \defeq 1 - e^{-qx}
  $.
  Indeed, observe
  \begin{align*}
    \rho_t^k I_k( \bnu, f_q )
      &= \int_{(0,\infty)^k} \brak[\Big]{ f_q\paren[\Big]{\sum_{i=1}^k x_i} - \sum_{i=1}^k f_q(x_i) } \, d\bnu_t^k\\
      &= \int_{(0, \infty)^k} \brak[\Big]{ 1 - \prod_{i=1}^k e^{-q x_i} -\sum_{i=1}^k \paren[\big]{ 1 - e^{-q x_i} } } \, d\bnu_t^k\\
   &= \rho_t^k - (\rho_t-\varphi(q, t))^k - k \rho_t^{k-1}\varphi(q, t) \,.
  \end{align*}
  Consequently
  \begin{align*}
    \sum_{k \geq 2} R_k I_k( \bnu, f_q )
      &= \sum_{k \geq 2}
	  \frac{(-1)^k \Psi^{(k)}(\rho_t)}{k!}
	  \brak[\Big]{
	    \rho_t^k - (\rho_t-\varphi(q, t))^k - k \rho_t^{k-1}\varphi(q, t)
	  }
      \\
      &= \sum_{k \geq 0}
	  \frac{(-1)^k \Psi^{(k)}(\rho_t)}{k!}
	  \brak[\Big]{
	    \rho_t^k - (\rho_t-\varphi(q, t))^k - k \rho_t^{k-1}\varphi(q, t)
	  }
      \\
      &= \Psi(0) - \Psi( \varphi(q, t) ) + \varphi(q, t) \Psi'(0)
      = \varphi(q, t) \Psi'(0) + \partial_t \ip{\bnu}{f_q} \,,
  \end{align*}
  where second equality is true because the terms for $k = 0, 1$ are $0$, the third equality follows from the Taylor expansion of $\Psi$ about $\rho_t$, and the last equality follows from~\eqref{e:PhiEvol} and the fact that $\Psi(0) = 0$.

  This proves~\eqref{e:SmolSubC} is satisfied for test functions of the form $f_q$ above.
  The general case follows from an approximation argument the details of which are the same as in the proof of Theorem~3.2 in~\cite{IyerLegerEA15}.
\end{inparaenum}
  \end{proof}

  \subsection{Convergence of Galton-Watson processes.}\label{s:GW}
  Finally, we conclude this section with convergence results for the rescaled Galton-Watson processes.
  We rescale the population by a factor of $h_k$ and the reproduction time by a factor of $\tau_k$.
  Explicitly, the rescaled Galton-Watson process $Y^{(k)}$ is defined by
  \begin{equation}\label{e:YkDef}
    Y^{(k)}_{t}(x) = h_k X_{\floor{t/\tau_k}, k}\paren[\Big]{ \floor[\Big]{ \frac{x}{h_k} } } \,.
  \end{equation}
  We recall that the argument of the processes refers to the initial population and is usually suppressed.
%

  \begin{proposition}\label{p:GWConv}
    Make the same assumptions as in Proposition~\ref{p:BrMechConv}.

    \begin{asparaenum}[(i)]
      \item
	Suppose there exists a finite measure~$\kappa$ on $[0, \infty]$ such that $(\scalek \mu)$ L\'evy-converges (as in Definition~\ref{d:Lconv}) to~$\kappa$.
	Then the finite-dimensional distributions of the rescaled Galton-Watson processes $Y^{(k)}$ converge to those of a 
finite (sub)critical CSBP~$Z$.
	Further, the Bernstein transforms of $Y^{(k)}$'s converge pointwise to the Laplace exponent of $Z$.
	That is, if
	\begin{gather*}
	  \varphi_k(q, t)
	  \defeq \frac{1}{h_k}
	      \E \brak[\big]{
		1 - \exp\paren[\big]{-q Y^{(k)}_{t}(h_k)} 
	      }\,,
	  \\
	  \llap{\text{and}\qquad}
	  \varphi(q, t)
	  \defeq -\ln \E \brak{ \exp\paren{-q Z_t(1)} } 
	\end{gather*}
	then
	\begin{equation*}
	  \lim_{k \to \infty} \varphi_k(q, t)
	  = \varphi(q, t)\,,
	\end{equation*}
	for every $q \in \R^+$ and $t \geq 0$.

      \item
	Conversely, if the one-dimensional distributions of $Y^{(k)}$ converge to those of a non-negative process $Z$ such that $\P\set{Z_{t_0}(x_0) > 0 }> 0$ for some $x_0,t_0 > 0$,
then there exists a finite measure $\kappa$ on $[0, \infty]$ such that $(\scalek \mu)$ L\'evy-converges to~$\kappa$.
	Consequently, $Z$ may be chosen to be a finite (sub)critical CSBP.
    \end{asparaenum}
  \end{proposition}

  \begin{remark}\label{r:subcLim}
    Before presenting the proof, we momentarily pause to remark on criticality.
    Each process $Y^{(k)}$ is a critical Galton-Watson process, however, the limit need not be critical.
    Indeed, the branching mechanism of the limit is given by~\eqref{e:PsiLim} with $\alpha_\infty \geq 0$.
    This means that the limiting process $Z$ can either be \emph{critical} or \emph{subcritical}.
    The subcritical situation ($\alpha_\infty > 0$) arises precisely when the rescaled measures $\scalek\mu$  L\'evy-converge to a measure that has an atom at $\infty$.
    Physically this corresponds to the situation when the rescaled descendant distributions~$\scalek\mu$ have a negligible fraction of large families that contain a non-negligible fraction of the population.
In the continuum limit, this fraction of the population ends up in families of `infinite' size, 
and the total population observed in families of finite size decays in time at exponential rate $a_\infty$.
  \end{remark}

  \begin{proof}
    First assume~\eqref{e:MuLim} holds.
    By Proposition~\ref{p:BTConv} we know that~$\varphi_k$ converges to a function~$\varphi$ which satisfies~\eqref{e:PhiEvol}.
    We know from~\cite{IyerLegerEA15} (see also~\cite{BertoinLeGall06}) that the solution of~\eqref{e:PhiEvol} is the Laplace exponent of a (sub)critical CSBP~$Z$.
    To finish the proof we only need to show that the finite-dimensional distributions of $Y^{(k)}$ converge to that of $Z$.
    Let $x > 0$ be the initial population and fix $t > 0$.
    Define $N_k = \floor{x / h_k}$ and observe
    \begin{multline*}
      \E \paren[\Big]{
	\exp\paren[\big]{ -q Y^{(k)}_t(x) } 
      }
      =
      \E \paren[\Big]{
	\exp\paren[\big]{ -q Y^{(k)}_t(N_k h_k)} 
      }
      \\
      = \brak[\Big]{ \E \paren[\Big]{
	  \exp\paren[\big]{ -q Y^{(k)}_t(h_k) } 
	} }^{N_k}
      = \brak[\Big]{ 1 - h_k \varphi_k( q, t ) }^{N_k}
      \\
      \xrightarrow{k \to \infty} e^{-\varphi(q, t) x}
      = \E \paren[\Big]{
	\exp\paren[\big]{ -q Z_t(x) } 
      }
    \end{multline*}
    This shows that the Laplace transforms of $Y^{(k)}_t$ converge to the Laplace transform of $Z$, proving convergence of one-dimensional distributions.
    The convergence of higher dimensional distributions can now be obtained as in~\cite[page 280]{Lamperti67a}.

    For the converse, assume that the one-dimensional distributions of $Y^{(k)}$ converge to that of a process $Z$ with $\P\{Z_{t_0} > 0\} > 0$.
    Let
    $x > 0$, and $N_k = \floor{x / h_k}$, and observe convergence of one-dimensional distributions implies
    \begin{equation*}
      \E \paren[\Big]{
	\exp\paren[\big]{ -q Y^{(k)}_{t_0}(x) } 
      }
      \xrightarrow{k \to \infty} e^{-\varphi(q, {t_0}) x}
      = \E \paren[\Big]{
	\exp\paren[\big]{ -q Z_{t_0}(x) } 
      }\,.
    \end{equation*}
    Thus
    \begin{equation*}
      \lim_{k \to \infty}
	\brak[\Big]{ 1 - h_k \varphi_k( q, t_0 ) }^{N_k}
      \quad\text{exists and equals }
      e^{- \varphi(q, t_0) x}\,,
    \end{equation*}
    for some function $\varphi$.
    Clearly $\varphi(q, t_0)x$ must be the Laplace exponent of $Z_{t_0}(x)$, and further
    \begin{equation*}
      \lim_{k \to \infty} \varphi_k(q, t_0)
	= \lim_{k \to \infty}
	  \frac{1 - e^{-\varphi(q, t_0)x / N_k}}{h_k}
	= \varphi(q, t_0).
    \end{equation*}
    By our assumption on $Z$, we know $\varphi(q, t_0) > 0$ and we can apply Proposition~\ref{p:BTConv}.
    This will guarantee \eqref{e:MuLim} holds. By using part (i), we infer there is a
finite (sub)critical CSBP with the same one-dimensional distributions as $Z$. This completes the proof.
  \end{proof}

  As an immediate corollary, we prove a special case of a result in~\cite{BertoinLeGall06} showing that the L\'evy jump measure of certain CSBPs satisfy the Smoluchowski equation.
  \begin{corollary}\label{c:CSBPSmol}
    Let $Z$ be a finite (sub)critical CSBP whose branching mechanism satisfies Grey's condition~\eqref{e:Grey}, and $\nu_t$ be the L\'evy jump measure associated with $Z_t$.
    Then $\nu$ is a weak solution of the damped Smoluchowski equations~\eqref{e:SmolSubC}.
  \end{corollary}
  \begin{proof}
    We note first that there exists a sequence of processes $Y^{(k)}$ of the form~\eqref{e:YkDef} such that the finite-dimensional distributions of $Y^{(k)}$ converge to $Z$.
    While the existence of such a sequence is well known~\cite{Grimvall74,Li00} we will prove a much stronger result in Theorem~\ref{t:UniversalGW}, below.
    By Proposition~\ref{p:GWConv} we know that the Bernstein transforms of $Y^{(k)}$ converge pointwise to the Laplace exponent of $Z$.
    Consequently, by Corollary~\ref{clyLevyTripConv} the convergence~\eqref{e:NukLim} holds, where $\nu_t$ is the L\'evy jump measure of $Z_t$, and further~$\nu$ satisfies~\eqref{e:SmolSubC} as desired.
  \end{proof}

  \section{Comparison to general convergence criteria.}\label{s:Grimvall}

  Necessary and sufficient criteria for convergence of scaling limits of Galton-Watson processes in general were developed by Grimvall~\cite{Grimvall74}, and recently by Bansaye and Simatos~\cite{BansayeSimatos15}.
  Grimvall's original criteria involved convolution powers of  family-size distributions, shifted and scaled.  
  The (equivalent) convergence criteria of Bansaye and Simatos~\cite{BansayeSimatos15}, obtained in connection with their more general investigation of branching in time-varying environments, are simpler and only involve convergence of moments and tails of the family-size distribution.
  We present these criteria here and show that for critical Galton-Watson processes the Bansaye-Simatos criteria are equivalent to the L\'evy-convergence criterion in Proposition~\ref{p:GWConv}.

  \subsection{The Grimvall and Bansaye-Simatos convergence criteria}
  We begin by stating Grimvall's result using notation consistent with ours.
  \begin{theorem}[Grimvall~\cite{Grimvall74}, Theorems~3.1 and~3.3]\label{t:Grimvall}
Let $(h_k)$, $(\tau_k)$ be positive sequences converging to 0, 
such that $1 / (h_k \tau_k) \in \N$ for all $k$.
    Let $(\eta_k)$ be a sequence of probability measures of the form
    \begin{equation*}
      \eta_k
	= 
	  \sum_{j \geq -1} \hat \pi_k(j+1) \delta_{j h_k}\,.
    \end{equation*}
    \begin{enumerate}
      \item
	If $\eta_k^{*1 / \paren{ h_k \tau_k } }$ converges weakly to a probability measure $\eta$, then the finite-dimensional distributions of the scaled processes $Y^{(k)}$ defined in \eqref{e:YkDef} converge to those of a (possibly infinite) CSBP with an absorbing state $+\infty$.
      \item
	Conversely, if for every $t \in [0, 1]$ the random variables $Y^{(k)}_t$ converge weakly to a random variable $Y_t$, with $\P\set{Y_1 > 0} > 0$, then $\eta_k^{*1 / \paren{ h_k \tau_k } }$ converges weakly to a probability measure $\eta$.
    \end{enumerate}
  \end{theorem}

\begin{remark}\label{r:DiscLamperti}
  The reason 
the shifted measures $\eta_n$ are natural in this context is because they determine the law of the parent process of the Lamperti transform.
At the discrete level before scaling, this can be described as follows.
Let $\xi$ be a random variable with distribution $\hat \pi$, representing the number of descendants of a single individual.
Define the \emph{parent process}, $\bar X$, to be a random walk obtained by summing i.i.d.\ copies of $\xi - 1$, and a time change $\Theta$ defined recursively by
\begin{equation*}
  \Theta_0 = 0 \,,
  \quad
  \Theta_{n+1} \defeq \Theta_n + x + \bar X_{\Theta_n}\,.
\end{equation*}
The Lamperti transformation defines a process $X$ by
\begin{equation*}
  X_n \defeq x + \bar X_{\Theta_n}\,.
\end{equation*}
Observe
\begin{equation*}
  X_{n+1} - X_n = \bar X_{\Theta_{n+1}} - \bar X_{\Theta_n}\,,
\end{equation*}
and hence the increment $X_{n+1} - X_n$ has the same distribution as the sum of $X_n$ independent copies of $\xi - 1$ that are each independent of $X_n$.
Equivalently, $X_{n+1}$ has the same law as the sum of $X_n$ i.i.d.\ copies of $\xi$, that are each independent of $X_n$.
This implies $X_n$ is a Galton-Watson process with initial population $x$ and descendant distribution $\hat \pi$.
(This is the discrete analog of the Lamperti transform described in Section~\ref{s:CSBPintro}.)

Grimvall's convergence criterion stated above implies that the sequence of the parent processes $\bar X\supk$ generated by $\hat \pi_k$ converge weakly, after scaling, to a L\'evy process.
The limiting parent process will be the parent process of the CSBP as given by the Lamperti transform.
\end{remark}

In their investigation of scaling limits for Galton-Watson processes in varying environments, Bansaye and Simatos \cite{BansayeSimatos15} identified the following simplified criteria as equivalent to Grimvall's for the convergences in Theorem~\ref{t:Grimvall}. 
In the present situation, the criteria of Bansaye and Simatos 
(stated as Assumption A.1 in Appendix A of \cite{BansayeSimatos15}) may be stated as follows.

\begin{proposition}[Bansaye-Simatos~\cite{BansayeSimatos15}]\label{p:BS}
  The finite-dimensional distributions of the scaled processes $Y^{(k)}$ defined in \eqref{e:YkDef} converge to those of a (possibly infinite) CSBP with an absorbing state $+\infty$ if and only if there exists $\hat a\in\R$, $\hat b\ge0$ and a positive $\sigma$-finite measure $F$ on $(0,\infty)$ such that as $k\to\infty$,
  \begin{equation}\label{c:BS1us}
      \int_\R  \frac{x}{1+x^2}\,
  \frac{ d\eta_k(x)}{h_k\tau_k} \to \hat a \,, 
  \qquad
      \int_\R  \frac{x^2}{1+x^2}\, 
  \frac{ d\eta_k(x)}{h_k\tau_k} 
  \to \hat b \,.
  \end{equation}
  and 
  \begin{equation}\label{c:BS2us}
  \int_{[x,\infty)} 
  \frac{ d\eta_k(x)}{h_k\tau_k} 
  \to F([x,\infty)) \quad\mbox{for a.e. $x>0$.}
  \end{equation}
\end{proposition}

We remark that one can directly prove the equivalence of the criteria
\eqref{c:BS1us}-\eqref{c:BS2us} and Grimvall's convergence criterion in
Theorem~\ref{t:Grimvall}(1) using the 
classical theory of infinite divisibility and canonical measures
in \cite{Feller71}.

  \subsection{Equivalence to L\'evy-convergence in the critical case}

The main result of this section shows that under the criticality assumption~\eqref{e:Criticality}, the conditions~\eqref{c:BS1us} and~\eqref{c:BS2us} are equivalent to the L\'evy-convergence (Definition~\ref{d:Lconv}) of the rescaled L\'evy measures~$\scalek \mu$ determined from
family size distributions $\hat\pi_k$ as in \eqref{defMuht}.
What makes this a curious point is that 
L\'evy-convergence occurs only on the compactified positive half-line $[0,\infty]$, 
while the measures $\eta_k$ also have support on the negative half-line.
  \begin{proposition}\label{p:GrimvallEquiv}
    Let $(\hat \pi_k)$ be a sequence of measures satisfying the criticality assumption~\eqref{e:Criticality},
    $(h_k)$, $(\tau_k)$ be two positive sequences that converge to~$0$, and let $\scalek\mu$ be defined by~\eqref{defMuht}.
Then~\eqref{c:BS1us} and~\eqref{c:BS2us} are equivalent to the existence of a finite measure $\kappa$ on $[0,\infty]$ such that $(\scalek \mu)$ L\'evy-converges to $\kappa$.
  \end{proposition}
  \begin{proof} 
  1.  We note that by criticality and \eqref{defMuht},
  \[
  \frac{\hat\pi_k(0)}{\tau_k} = \frac1{\tau_k}\sum_{j\ge2} (j-1)\hat\pi_k(j) =
  \int_{(0,\infty)} d\scalek\mu(x)\,.
  \]
  Recalling the definition of $\eta_k$ in Theorem~\ref{t:Grimvall}, one computes that
\[
\int_{\R} x f(x) 
\frac{d\eta_k(x)}{h_k\tau_k} = 
\sum_{j\ge0} f(jh_k-h_k)(j-1)\frac{\hat\pi_k(j)}{\tau_k} = 
\int_{(0,\infty)} ( f(x-h_k)-f(-h_k) ) \,d\scalek\mu(x)
\]
whenever $x\mapsto xf(x)$ is bounded and continuous on $\R$.

2. Suppose $(\scalek\mu)$ L\'evy-converges to a finite measure $\kappa$ on $[0,\infty]$,
and associate a L\'evy triple $(\alpha_0,\alpha_\infty,\mu)$ with $\kappa$ by \eqref{e:KappaMu}.
Then taking $xf$ to approximate the characteristic function of $[z,\infty)$,
we find that for a.e. $z>0$, 
\begin{equation}\label{rel:F}
 \int_{[z,\infty)} \frac{d\eta_k(x)}{h_k\tau_k} \to 
 F([z,\infty)) = \int_{[z,\infty)} \frac{d\mu(x)}x \,.
\end{equation}
Further, taking $f$ to be $x^p/(1+x^2)$ for $p=0,1$ we find that as $k\to\infty$,
\[
g_k(x) = \frac{f(x-h_k)-f(-h_k)}{x\wedge1} 
= \frac1{x\wedge1}\int_0^x f'(z-h_k)\,dz \to  \frac{f(x)-f(0)}{x\wedge1}
\]
uniformly on $[0,\infty]$. Hence,
because the measures $(x\wedge1)\,d\scalek\mu(x)$ are uniformly bounded and converge
weak-star to $\kappa$ on $[0,\infty]$,
\begin{align} \label{rel:a}
& \hat a_k \defeq \int_{\R} \frac{x}{1+x^2} \frac{d\eta_k(x)}{h_k\tau_k} \to 
\hat a =  -\alpha_\infty - \int_{(0,\infty)} \frac{x^2}{1+x^2}\,d\mu(x)\,,
\\ \label{rel:b}
& \hat b_k \defeq \int_{\R} \frac{x^2}{1+x^2} \frac{d\eta_k(x)}{h_k\tau_k} \to 
\hat b = \alpha_0+ \int_{(0,\infty)} \frac{x}{1+x^2}\,d\mu(x)\,.
\end{align}
Thus both conditions \eqref{c:BS1us} and \eqref{c:BS2us} hold.

3.  Conversely, suppose \eqref{c:BS1us} and \eqref{c:BS2us} hold.
Taking $f(x)=(x-1)/(1+x^2)$ and supposing $h_k\in(0,1)$,
we find $g_k(x)\ge\frac12$ for all $x$ in the interval
$[2h_k,\infty)$ containing the support of $\scalek\mu$,
and that as $k\to\infty$,  
\[
\int_{(0,\infty)} g_k(x)(x\wedge1)\,d\scalek\mu(x) = \hat b_k-\hat a_k \to \hat b-\hat a \,.
\]
It follows that the measures $(x\wedge 1)\,d\scalek\mu(x)$ 
are uniformly bounded, hence weak-star compact on $[0,\infty]$. 
For any subsequential limit $\kappa$
on $[0,\infty]$ with associated L\'evy triple $(\alpha_0,\alpha_\infty,\mu)$,
the relations \eqref{rel:F}, \eqref{rel:a}, \eqref{rel:b} hold, 
and these uniquely determine the triple.  Hence the whole sequence $(\scalek\mu)$
L\'evy-converges.
\end{proof} 



%
%
%
  \part{Universality in Galton-Watson process and CSBPs.}\label{p:univ}
  \section{Universal Galton-Watson family size distributions}\label{s:Univ}%
\label{s:universalGW}

%

The following result is motivated by the existence of universal eternal solutions
of Smoluchowski's coagulation equations \cite{MenonPego08}, which in turn
was inspired by Feller's account of \emph{Doeblin's universal laws} 
in classical probability theory \cite[Section XVII.9]{Feller71}.

\begin{theorem}\label{t:UniversalGW}
There exists a Galton-Watson process $X$ with family-size distribution $\hat\pi\colon\N\to [0,\infty)$ and sequences $(h_n)$, $(\tau_n)\to0$, with the following property:
For any (sub)critical branching mechanism $\Psi$ taking the form in \eqref{e:PsiLim}, there is a subsequence along which the rescaled processes~$Y^{(k)}$ (defined by rescaling $X$ as in equation~\eqref{e:YkDef}) converge to a CSBP with branching mechanism $\Psi$.

Moreover, the set of such family-size distributions $\hat\pi$ is dense in the set of all probability measures on $\nplus$ with the weak-$\star$ topology.
\end{theorem}

Much of the technical basis that we need to prove this result can be 
inferred directly from Section~7 in \cite{MenonPego08}. 
The terminology used in \cite{MenonPego08}, however, is substantially different
from our terminology which is based on~\cite{SchillingSongEA10} and Appendix~\ref{s:CT}, below.
Thus for the readers convenience we provide a self-contained treatment here.



We begin by constructing a ``universal'' L\'evy triple, in the sense that any other L\'evy triple can be obtained as a suitable scaling limit.
In order to make this precise, we need to define the right notions of convergence and scaling of L\'evy triples.

The right notion of convergence of L\'evy triples is a generalization of L\'evy-convergence 
as introduced earlier in Definition~\ref{d:Lconv}.
Namely, to each L\'evy triple $\lambda \defeq (\alpha_0, \alpha_\infty, \mu)$ we associate~$\kappa$, a finite measure on $[0, \infty]$, by 
\begin{equation}\label{e:kappaDef}
  d\kappa(x) = a_0 \, d\delta_0(x) + a_\infty \, d\delta_\infty(x) + (x \varmin 1) \, d\mu(x) \,.
\end{equation}
Now convergence of L\'evy triples is defined using weak-$\star$ convergence of the associated $\kappa$-measures.
\begin{definition}\label{d:LtConv0}
  Let $(\lambda_k)$ be a sequence of L\'evy triples, and $\kappa_k$ the associated measures defined as in~\eqref{e:kappaDef} above.
  We say $(\lambda_k)$ converges if the sequence $(\kappa_k)$ converges weak-$\star$ in the spaces of finite measures on $[0, \infty]$.
\end{definition}
A more detailed account of this appears in Appendix~\ref{s:CT} below.
This appendix also contains variants of certain classical continuity theorems for Laplace transforms that are used throughout this section but are not widely known.


Next we define the scaling of L\'evy triples, as follows.
Given a L\'evy triple $\lambda=(\alpha_0,\alpha_\infty,\mu)$ and $b, c > 0$, define the rescaled L\'evy triple $\lambda^{b,c}$ by
\begin{equation*}
  \lambda^{b,c} \defeq (\alpha_0^{b, c}, \alpha_\infty^{b,c}, \mu^{b,c})\,,
\end{equation*}
where
\begin{equation}\label{e:ScaleLambda}
\alpha_0^{b,c} = cb\inv\,\alpha_0, \quad
\alpha_\infty^{b,c} = c\,\alpha_\infty,\quad
\text{and}\quad
d\mu^{b,c}(x) = c\,d\mu(bx) \,.
\end{equation}
We will see in Section~\ref{s:LinUniv}, below, that this naturally corresponds to a dilational scaling of CSBPs.
With this, we can define the aforementioned notion of universality.

\begin{definition}\label{d:universalLevy}
  Let $\lambda_\star$ be a L\'evy triple and $(b_k)$, $(c_k)$ be two sequences that converge to infinity.
  We say $\lambda_\star$ is a \emph{universal L\'evy triple} with sequences $(b_k)$, $(c_k)$ if for any L\'evy triple $\lambda$ we have
  \begin{equation}\label{e:UnivConv}
    \lambda_\star^{b_{k}, c_{k}} \to \lambda
    \quad\text{as $k \to \infty$ along some subsequence.}
  \end{equation}
\end{definition}

Our next lemma shows that 
this notion of universality is completely determined by the tail of $\lambda_\star$.
\begin{lemma}\label{l:UnivTail}
  Let $\lambda_\star = (0, 0, \mu_\star)$, be a universal L\'evy triple with sequences $(b_k)$ and $(c_k)$.
  Let $\alpha_0 \geq 0$, $R > 0$, $\mu$ be any L\'evy measure and define the L\'evy measure $\nu_\star$ by
  \begin{equation*}
    \nu_\star(A) = \mu(A \cap (0, R] ) + \mu_\star( A \cap (R, \infty) ) \,.
  \end{equation*}
  Then the L\'evy triple $(\alpha_0, 0, \nu_\star)$ is also universal with the sequences $(b_k), (c_k)$.
\end{lemma}
\begin{remark*}
  Lemma~\ref{l:UnivTail} is still true if finitely many terms of the sequences $(b_k), (c_k)$ are arbitrarily changed.
\end{remark*}

To prove Lemma~\ref{l:UnivTail} it helps to introduce left and right distribution functions
as follows.

\begin{definition}
  Given a L\'evy triple $\lambda$, we define the \emph{pair of left and right distribution functions} associated with $\lambda$ to be $(\kappa_L, \kappa_R)$ defined by
  \begin{equation*}
    \kappa_L(x) \defeq \alpha_0 + \int_{(0, x]} z \, d\mu(z)\,,
\quad 
    \kappa_R(x) \defeq \alpha_\infty + \int_{(x, \infty)} \, d\mu \,.
  \end{equation*}
\end{definition}

  The reason we introduce these functions is because of a variant of a classical continuity theorem: pointwise convergence (almost everywhere, or at points of continuity) of the pair of left and right distribution functions is equivalent to convergence of the associated L\'evy triples.
  Since the proof in this form is not readily available in the literature, we provide a proof in the appendix (Theorem~\ref{t:LRconv}).
  We can now prove Lemma~\ref{l:UnivTail}.

\begin{proof}[Proof of Lemma~\ref{l:UnivTail}]
  Let $\lambda$ be any L\'evy triple and choose a subsequence for which the convergence in~\eqref{e:UnivConv} holds.
  We claim that along this subsequence we must have $(c_k / b_k) \to 0$.
  Once this is established, the lemma immediately follows from the fact that
  \begin{equation*}
    \paren[\big]{
      c_k b_k\inv \alpha_0, 0,
      \abs[\big]{\mu_\star^{b_k, c_k} - \nu_\star^{b_k, c_k} }
    } \to 0
  \end{equation*}
  in the topology of L\'evy triples.

  To show $(b_k / c_k) \to 0$ let $(\kappa_{\star L}, \kappa_{\star R})$ be the pair of left and right distribution functions associated with $\lambda_\star$.
  Under scaling note
  \begin{equation}\label{e:KappaScale0}
    \kappa_{\star L}^{b_k,c_k}(x) = \frac{c_k}{b_k}\kappa_{\star L}^{}(b_k x)\,,
    \quad\text{and}\quad 
    \kappa_{\star R}^{b_k,c_k}(x) = c_k^{}\,\kappa_{\star R}^{}(b_k x)\,,
  \end{equation}
  Now by monotonicity of $\kappa_{\star L}$ and universality of $\lambda_\star$ it follows that $\kappa_{\star L}(y) \to \infty$ as $y \to \infty$.
  Moreover, \eqref{e:UnivConv} and the continuity theorem (Theorem~\ref{t:LRconv}) imply that $\kappa_{\star L}^{b_k, c_k}(x)$ converges (along the chosen subsequence) to a finite limit for some $x$.
  This forces $(c_k / b_k) \to 0$ along the chosen subsequence, finishing the proof.
\end{proof}

The main idea behind the proof of Theorem~\ref{t:UniversalGW} is the existence of many universal L\'evy triples.
This is our next result.

\begin{proposition}\label{p:UniversalLT}
  There exist sequences $(b_k)$, $(c_k)$ such that
  the set of L\'evy triples of the form $(0, 0, \mu_\star)$ which are universal with respect to $(b_k)$, $(c_k)$
  is dense in the space of \emph{all} L\'evy triples.
\end{proposition}

We postpone the proof of Proposition~\ref{p:UniversalLT} until the proof of Theorem~\ref{t:UniversalGW} is complete.

\begin{proof}[Proof of Theorem~\ref{t:UniversalGW}]
  Using Proposition~\ref{p:UniversalLT} we choose a L\'evy triple $\lambda_\star = (0, 0, \mu_\star)$  that is universal with sequences $(b_k)$, $(c_k)$.
  Using Lemma~\ref{l:UnivTail} we can, without loss of generality, assume
  \begin{equation*}
    \int_{(0, \infty)} d\mu_\star(x) < 1 \,.
  \end{equation*}

  Let
  \begin{equation}
  \tau_k = c_k\inv, \qquad h_k = b_k\inv\,,
  \end{equation}
  %
  and define a family size distribution $\hat\pi(j)$, $j\geq0$, by
  \begin{align}\label{e:DefHatpi}
  & \hat\pi(0)= 1-\int_\rplus\,d\mu_\star(x)\,,
  \\
  &\hat\pi(j) = \frac1{j-1}\int_{(j-2,j-1]} d\mu_\star(x)\,,\quad j\geq2 \,,
  \end{align}
  and set $\hat\pi(1)$ so that $\sum_{j\geq0}\hat\pi(j)=1$. 
  Define the coarse-grained L\'evy measure $\hat\mu$
  by
  \begin{equation}\label{e:DefHatmu}
  d\hat\mu(x) = \sum_{j\geq2} (j-1)\hat\pi(j)\,d\delta_j(x)
  = \sum_{j\geq2} \paren[\Big]{\int_{(j-2,j-1]} d\mu_\star(x)} \, d\delta_j(x) 
  \,.
  \end{equation}
  We claim that the L\'evy triple $\hat \lambda \defeq (0, 0, \hat \mu)$ is also universal.
  Once universality of $\hat \lambda$ is established, subsequential convergence of $Y^{(k)}$ follows immediately from Proposition~\ref{p:GWConv}.

  To prove universality of $\hat \lambda$, let $\lambda = (\alpha_0, \alpha_\infty, \mu)$ be an arbitrary L\'evy triple.
  By universality of $\lambda_\star$ we have
  \begin{equation*}
    \lambda_\star^{b_{k}, c_{k}} \to \lambda
    \quad\text{as $k \to \infty$ along some subsequence.}
  \end{equation*}
  For brevity, we use $k$ to index the above subsequence.

  Let $(\kappa_L, \kappa_R)$, $(\kappa_{\star L}, \kappa_{\star R})$ and $(\kappa_{\star L,k}, \kappa_{\star R, k})$ be the pairs of left and right distribution functions associated to the L\'evy triples~$\lambda$, $\lambda_\star$ and $\lambda^{b_k, c_k}_\star$ respectively.
  Since, $\lambda_\star^{b_k, c_k} \to \lambda$ by choice of our subsequence, we must have
  \begin{equation}\label{e:LimKappaRLstar}
  \frac{h_k}{\tau_k} \kappa_{\star L}\left(\frac x{h_k}\right) \to\kappa_L(x)\,,
  \qquad
  \frac1{\tau_k} \kappa_{\star R}\left(\frac x{h_k}\right) \to\kappa_R(x)\,,
  \end{equation}
  at all points of continuity
  along the same subsequence.
  We claim that, along the same subsequence and at the same points,
  we must have
  \begin{equation}\label{e:LimHatkappas}
  \frac{h_k}{\tau_k} \hat\kappa_L\left(\frac x{h_k}\right) \to\kappa_L(x)\,,
  \qquad
  \frac1{\tau_k} \hat\kappa_R\left(\frac x {h_k}\right) \to\kappa_R(x)\,,
  \end{equation}
  where $(\hat \kappa_L, \hat \kappa_R)$ are the pair of left and right distribution functions associated to $\hat \lambda$.

  To see this, observe that for any $z\in(2,\infty)$,
  \[
  \int_{(z,\infty)}d\mu_\star(x)
  \leq \sum_{j>z} \int_{(j-2,j-1]} d\mu_\star(x)
  \leq \int_{(z-2,\infty)}d\mu_\star(x)\,,
  \]
  which means
  \[
  \kappa_{\star R}(z) \leq \hat\kappa_R(z) 
  \leq \kappa_{\star R}(z-2) \,.
  \]
  Then for any two points of continuity $0<x_-<x$ for $\kappa_R$,
  whenever $h_k$ is small enough we have 
  $h_k\inv x_-< h_k\inv x-2<h_k\inv x$, whence
  \begin{equation}
   \kappa_R(x) \leq
    \liminf \frac1{\tau_k}\hat\kappa_R\left(\frac x{h_k}\right) 
    \leq 
   \limsup \frac1{\tau_k}\hat\kappa_R\left(\frac x{h_k}\right) 
  \leq
  \kappa_R(x_-) \,.
   \end{equation}
  Thus the second limit in \eqref{e:LimHatkappas} follows.

  To establish the first limit, observe that for any $z>2$,
  \[
  \int_{(0,z-1]} x\,d\mu_\star(x) \leq 
  \sum_{2\leq j\leq z} j\int_{(j-2,j-1]}d\mu_\star(x) \leq
  \int_{(0,z]}(x+2) d\mu_\star(x) \,,
  \]
  and this entails
  \[
  \kappa_{\star L}(z-1) \leq \hat\kappa_L(z) \leq \kappa_{\star L}(z) + 2\,.
  \]
  Using the argument in Lemma~\ref{l:UnivTail} we see $b_k / c_k = h_k / \tau_k \to 0$.
  Consequently, for any two points of continuity $x_-<x$ for $\kappa_L$, 
  as above we infer
  \begin{equation}
   \kappa_L(x_-) \leq
    \liminf \frac {h_k}{\tau_k} \hat\kappa_L\left(\frac x{h_k}\right) 
    \leq 
   \limsup \frac{h_k}{\tau_k}
   \left( \hat\kappa_L\left(\frac x{h_k}\right) +2\right)
  \leq
  \kappa_L(x) \,.
  \end{equation}
  The second limit in \eqref{e:LimHatkappas} follows, and this proves universality of the coarse grained measure~$\mu_\star$.
  This finishes the existence part of Theorem~\ref{t:UniversalGW}.

  To prove density, we use Lemma~\ref{l:UnivTail} to note that any probability measure on $\nplus$ with the same tail at $\hat \pi$ is also universal.
  Since all such probability measures are dense in the space of all probability measures on $\nplus$, we obtain density.
  This finishes the proof of Theorem~\ref{t:UniversalGW}.
  %
\end{proof}

It remains to prove Proposition~\ref{p:UniversalLT}.
Fix a sequence $c_k\to\infty$ that satisfies 
\begin{equation}\label{e:Cseq}
\sum_{j=1}^\infty \frac{j}{c_j} <1 \,, \qquad 
c_k \sum_{j>k}  \frac{j}{c_j} \to 0 \quad\text{as }k\to\infty \,.
\end{equation}
For example, $c_k= c \cramped{e^{k^2}}$ works for small enough $c>0$.
  The main tool used in the proof of Proposition~\ref{p:UniversalLT} is the following ``packing lemma'', which is the analog of Lemma~7.2 in~\cite{MenonPego08}.

\begin{lemma}[Packing lemma]\label{lem:uni1}
Let $\lambda_k = (0,0,\mu_k)$ be a sequence of L\'evy triples,
let $\Phi_k$ be the corresponding Bernstein transforms given by 
\[
\Phi_k(q) = \int_\rplus \paren[\big]{1-e^{-qx}}\,d\mu_k(x)\,,
\]
and assume that 
\begin{equation} \label{i:levybd}
\int_\rplus d\mu_k(x) \leq k \,,\quad\text{for all } k \in \N \,.
\end{equation}
Then there exists a sequence $b_k\to\infty $ 
such that the series
\begin{equation}\label{eq:Phiseries}
\Phi_\star(q) \defeq \sum_{j=1}^\infty  c_j\inv \Phi_j (b_jq)
\end{equation}
converges for each $q>0$ to a Bernstein function with the 
following property:
\begin{equation}\label{eq:Phiscale}
c_k \Phi_\star(b_k\inv q) - \Phi_k(q) \to 0 
\quad\mbox{as $k\to\infty$, \quad for all $q>0$\,.}
\end{equation}
The function $\Phi_\star$ is the Bernstein transform of a L\'evy triple
of the form $\lambda_\star=(0,0,\mu_\star)$, 
where $\mu_\star$ is a finite measure on $(0,\infty)$, with $\int_{\rplus}d\mu_\star(x)<1$. 
\end{lemma}

\begin{proof}
As $\Phi_j(b_jq)\leq \Phi_j(\infty)\leq j$ for all $q>0$, the series
\eqref{eq:Phiseries} is uniformly bounded and converges for each $q$. 
We estimate the quantity in \eqref{eq:Phiscale} in two parts, to 
show how $b_k$ can be chosen. Note
\[
  \abs[\bigg]{ \Phi_k(q) - c_k \sum_{j=k+1}^\infty c_j\inv\Phi_j(b_k\inv b_j q) }
\leq c_k \sum_{j=k+1}^\infty jc_j\inv  \to0
\]
as $k\to\infty$, regardless of what $b_k$ is. Then because $\Phi_j(q)\to0$
as $q\to0$ for each $j$,
we may choose $b_k$ (inductively) so large that 
\[
c_k \sum_{j=1}^{k-1} c_j\inv\Phi_j(b_k\inv b_j q) < \frac1k\,.
\]
Then \eqref{eq:Phiscale} follows.  

The limit function $\Phi_\star$ is the Bernstein transform
of some L\'evy triple $\lambda_\star = (a_0,a_\infty,\mu_\star)$ 
by Theorem~\ref{t:Cont}. 
Because $\Phi_j(0^+)=0$, we infer 
\begin{equation*}
  \Phi_\star(0^+)\leq \sum_{j=k+1}^\infty jc_j\inv
\end{equation*}
for all $k$, hence $a_\infty=\Phi_\star(0^+)=0$. Furthermore, 
\[
\Phi_\star(\infty) \leq \sum_{j\geq 1}jc_j\inv <1\,,
\]
whence $a_0=0$ and $\int_\rplus d\mu_\star(x)<1$.
\end{proof}

Next we show that L\'evy triples of the form $(0, 0, \mu)$ are dense in the space of all L\'evy triples.
This is analogous to Lemma 7.3 in \cite{MenonPego08}.
\begin{lemma}\label{lem:uni2}
Let $\lambda=(a_0,a_\infty,\mu)$ be a L\'evy triple.
Then there is a sequence of measures $\mu_k$
that satisfy \eqref{i:levybd}, such that 
the triples
\[
\lambda_k=(0,0,\mu_k)\to \lambda \quad\text{ as } k\to\infty \,.
\]
\end{lemma}
\begin{proof}
For the proof it suffices to consider $\lambda$ of the form $\lambda=(0,0,\mu)$,
because these are dense in the space of L\'evy triples. 
But then there exist $\eps_k\to0$ such that 
$\mu_k \defeq \one_{\set{x\geq \eps_k}}\mu$ satisfies
$\int_\rplus d\mu_k \leq \frac k3$. Evidently, 
$\lambda_k=(0,0,\mu_k)\to\lambda$
due to Definition~\ref{d:LtConv0}.
\end{proof}

With this we can prove Proposition~\ref{p:UniversalLT}.
\begin{proof}[Proof of Proposition~\ref{p:UniversalLT}]
  Because the space of finite measures on $[0,\infty]$ with the weak-$\star$
  topology is separable,
  the same is true for the space of L\'evy triples.
  Thus we can choose a sequence of L\'evy triples $(\bar\lambda_n)$ so that \emph{every} L\'evy triple $\lambda$ whatsoever is a limit of $(\lambda_k)$ along some subsequence.  

  Partition the integers into infinitely many subsequences and, using Lemma~\ref{lem:uni2}, select measures $\mu_k$ satisfying \eqref{i:levybd} such that $\lambda_k=(0,0,\mu_k)\to\bar\lambda_n$ along the $n^\text{th}$ subsequence. 
  Construct a sequence~$b_k$, a Bernstein function $\Phi_\star$ and its associated L\'evy triple $\lambda_\star$ using Lemma~\ref{lem:uni1}.
  We claim $\lambda_\star$ is universal with sequences $(b_k), (c_k)$.

  To see this, let $\Phi_\star^{b_k, c_k}$ be the Bernstein transform of the rescaled L\'evy triple $\lambda_\star^{b_k, c_k}$.
  Observe
  \begin{equation}\label{e:ScaleBernstein}
    \Phi_\star^{b_k, c_k}(q) = c_k \Phi_\star\paren[\Big]{ \frac{q}{b_k} } \,.
  \end{equation}
  Now given any L\'evy triple $\lambda$, choose a subsequence (indexed by $k$) along which $(\bar \lambda_k) \to \lambda$.
  If $\bar \Phi_k$ is the Bernstein transform of $\bar \lambda_k$, then the continuity theorem for Bernstein transforms (Theorem~\ref{t:Cont}) %
  guarantees $(\bar \Phi_k) \to \Phi$ pointwise.
  By~\eqref{eq:Phiscale} this implies $(\Phi^{b_k, c_k}_\star) \to \Phi$, and by the continuity theorem again we have $(\lambda_\star^{b_k, c_k}) \to \lambda$ along this subsequence.
  This finishes the proof.
\end{proof}

\section{Linearization and universality for critical CSBP}\label{s:LinUniv}


The dilational form of the scaling relations~\eqref{e:ScaleLambda}, \eqref{e:KappaScale0} and~\eqref{e:ScaleBernstein} hints at exact scaling relations that hold for the limiting CSBPs, which we develop in this section.
These relations establish that the \emph{nonlinear dynamics of CSBPs becomes linear and purely dilational} when expressed in terms of the L\'evy triple that represents the branching mechanism.
The map from L\'evy triples to finite (sub)critical CSBPs is bicontinuous
due to Theorem~\ref{t:CSBPCont} below, and this reduces the study
of scaling dynamics for these CSBPs to the study of scaling limits of 
dilation maps.  We use this correspondence here to prove
the existence of universal critical CSBPs whose subsequential scaling limits
include all possible finite (sub)critical CSBPs. 

Our results on linearization here are strongly analogous
to the results of \cite{MenonPego08} that showed that scaling dynamics
on the scaling attractor for solvable Smoluchowski equations
becomes linear and dilational in terms of a L\'evy-Khintchine representation.
The definition of the scaling attractor in \cite{MenonPego08} as the collection
of all limits of rescaled solutions was motivated by analogy to the notion of
infinite divisibility in probability. The analogy between infinite divisibility
and CSBPs in branching processes has long been evident, 
at least since the work of Grimvall. 
But this classical analogy does not appear to explain why 
the nonlinear dynamics of CSBPs (or the scaling attractor of Smoluchowski equations)
becomes linear in terms of L\'evy-Khintchine representations.

\subsection{Linearization of renormalized CSBP dynamics}

For finite (sub)critical CSBPs, the linearization property we are
talking about is stated in the following result.
It is actually a simple consequence of the known dynamics of the Laplace
exponent in terms of the L\'evy-Khintchine representation of the branching mechanism.

\begin{proposition}\label{p:Lin}
  Let $Z(x,t)$ be a CSBP with 
Laplace exponent $\vp(q,t)$ and branching mechanism
$\Psi$.  Let $b,c>0$. Then the rescaled CSBP given by 
\[
\tilde Z(x,t) = b\inv Z(bx,ct) 
\]
satisfies $\tilde Z(x,0)=x$
and has Laplace exponent
$\tilde\vp$ and branching mechansim $\tilde\Psi$ with 
\begin{equation}\label{tiPhiPsi}
\tilde\vp(q,t) = b\,\vp(b\inv q,ct)\,, \qquad \tilde\Psi(q) = cb\,\Psi(b\inv q)\,.
\end{equation}
If the branching mechanism $\Psi$ is (sub)critical, and 
given by \eqref{e:PsiLim} in terms of L\'evy triple
$(\alpha_0,\alpha_\infty,\mu)$, 
then the branching mechanism $\tilde\Psi$ is determined similarly by the 
L\'evy triple $(\tilde\alpha_0,\tilde\alpha_\infty,\tilde\mu)$ where
\begin{equation}\label{e:ScaleMuTilde}
\tilde\alpha_0 = cb\inv\,\alpha_0\,, \quad
\tilde\alpha_\infty = c\,\alpha_\infty\,,\quad
d\tilde\mu(x) = c\,d\mu(bx)\,,
\end{equation}
and this corresponds to left and right distribution functions given by
the scaling relations
\begin{equation}\label{e:KappaScale}
\tilde\kappa_L(x) = cb\inv\,\kappa_L(bx)\,,\quad
\tilde\kappa_R(x) = c\,\kappa_R(bx)\,.
\end{equation}
\end{proposition}
\begin{remark}
  The scaling~\eqref{e:ScaleMuTilde} and~\eqref{e:KappaScale} above are exactly the same as the scaling relations~\eqref{e:ScaleLambda} and~\eqref{e:KappaScale0} in Section~\ref{s:Univ}.
\end{remark}

\begin{proof}
The first relation in \eqref{tiPhiPsi} follows from the computation
\[
e^{-x\tilde\vp(q,t)} = \E( e^{-q\tilde Z(x,t)}) = \E( e^{-qb\inv Z(bx,ct)})
=e^{-xb\vp(b\inv q,ct)}\,,
\]
and the second relation follows by considering the ODE 
\eqref{e:PhiEvol1} satisfied by $\vp$ and the corresponding one satisfied
by $\tilde\vp$ at $t=0$, recalling that $\vp(q,0)=q=\tilde\vp(q,0)$.
The relations involving L\'evy triples and associated left and right distribution
functions follow by comparing the respective representations for $\Psi$ and $\tilde\Psi$ from \eqref{e:PsiLim}, and using the definitions in \eqref{def:kappaLR}. 
\end{proof}

For the study of long-time scaling limits of CSBPs, we should take $b,c$
as functions of $t$ with $c(t)\to\infty$
or as sequences $(b_k)$, $(c_k)$ with $c_k\to\infty$. 
This study can be reduced to the study of scaling limits
of the purely dilational relations in \eqref{e:KappaScale} 
due to the following continuity theorem. 
The proof of this theorem is essentially similar to (but simpler than)
the proof of Propositions~\ref{p:BTConv} and \ref{p:GWConv} for scaling limits
of Galton-Watson processes. 

For brevity in expressing our result, let us say that 
the L\'evy triple $\lambda=(\alpha_0,\alpha_\infty,\mu)$
in \eqref{e:PsiLim} \textit{generates} the branching mechanism $\Psi$
and the corresponding (finite, subcritical) CSBP $Z$. 

\label{ppCSBPCont}
%

\begin{theorem}[Continuity theorem for finite (sub)critical CSBPs]
\label{t:CSBPCont}
Let $(Z_k)$ be a sequence of finite (sub)critical CSBPs
generated by L\'evy triples 
$\lambda_k=(a^k_0,a^k_\infty,\mu_k)$. 

(i) Suppose that $(\lambda_k)$ converges to some L\'evy triple
$\lambda$ in the sense of Definition~\ref{d:LtConv0}. 
Then the finite-dimensional distributions of $Z_k$ converge
to those of the finite (sub)critical CSBP $Z$ generated by $\lambda$.

(ii) Conversely, suppose the one-dimensional distributions of $Z_k$ converge
to those of some finite process $Z$ such that $\P\{Z_{t_0}(x_0)>0\}>0$
for some $x_0,t_0>0$. Then the sequence $\lambda_k$ converges to some
L\'evy triple $\lambda$, and $Z$ may be taken as the finite (sub)critical
CSBP generated by $\lambda$.
\end{theorem}

\begin{proof}
To prove (i), suppose $\lambda_k$ converges to $\lambda$.
Then, due to the second continuity theorem~\ref{t:LRconv},
the branching mechanisms $\Psi_k$ generated by $\lambda_k$ 
converge pointwise to the branching mechanism $\Psi$ generated by $\lambda$.
Due to the representation formula~\eqref{def:Bprimitive} 
and an argument using Montel's theorem 
analogously to the proof of Proposition~\ref{p:BrMechConv},
the convergence occurs locally uniformly for every derivative. 
Then it is straightforward to show that because the Laplace exponents
$\vp_k$ of $Z_k$ satisfy
\begin{equation}\label{e:phikevol}
\partial_t\vp_k(q,t) = -\Psi_k(\vp_k(q,t)) \,,\qquad \vp_k(q,0)=q \,,
\end{equation}
one has $\vp_k(q,t)\to\vp(q,t)$ as $k\to\infty$ for each $q>0$, $t\geq0$,
where $\vp$ satisfies \eqref{e:PhiEvol}. 
In particular, $\vp$ is the Laplace exponent of the CSBP $Z$ generated by $\lambda$.
This implies that for each $x,t>0$, the Laplace tranform of $Z_k(x,t)$
converges to that of $Z(x,t)$. 
This proves the convergence of one-dimensional distributions.
The convergence of finite-dimensional distributions follows
using arguments
similar to Lamperti~\cite[page 280]{Lamperti67a}.

To prove (ii), assume the one-dimensional distributions
of $Z_k$ converge to those of some finite process $Z$ such that
$\P\{Z(x_0,t_0)>0\}>0$ for some $x_0,t_0>0$.
Then because the Laplace transforms converge, i.e., because
\[
e^{-x \vp_k(q,t)}=
\E(e^{-qZ_k(x,t)}) \to \E(e^{-qZ(x,t)})
\]
for each $x,t>0$, we have that
\[
\vp(q,t) \defeq \lim_{k\to\infty} \vp_k(q,t)  
\]
exists for each $q,t>0$. Taking $t=t_0$ we find that $\vp(q,t_0)>0$
for all $q>0$.  For small enough $q>0$, then, $2q<\vp(\hat q,t_0)$
for some $\hat q$, and because $\Psi_k$ is increasing and convex,
for large enough $k$ we find $2q<\vp_k(\hat q,t_0)$, 
hence by integrating~\eqref{e:phikevol} we find
\[
\Psi_k(2q) \leq \frac1{t_0}\int_0^{t_0} \Psi_k(\vp_k(\hat q,t))\,dt \leq 
\frac{\hat q}{t_0}\,.
\]
Hence with $\kappa_k$ determined from $\lambda_k$ as in~\eqref{e:kappaDef}
and $g_q(x)$ defined by \eqref{d:gq}, we find
\[
\ip{g_q}{\kappa_k} = \Psi_k'(q) \leq q\Psi_k(2q) \leq \frac{ q\hat q}{t_0} \,.
\]
As in the proof of the continuity theorem~\ref{t:Cont} (converse part), 
this implies
that $\ip{1}{\kappa_k}$ is bounded. Hence the sequence of L\'evy measures
$(\lambda_k)$ is precompact. Along any subsequence that converges
to some L\'evy triple $\lambda$, we may invoke part (i) to assert that 
the finite-dimensional distributions of $Z_k$ converge to those of the
finite (sub)critical CSBP generated by $\lambda$.  The one-dimensional
distributions of this CSBP are then the same as those of $Z$,
independent of the subsequence. It follows $\lambda$ is unique,
and the whole sequence $(\lambda_k)$ converges.
\end{proof}

%
%
%

\subsection{Existence of universal critical CSBPs}

The continuity theorem~\ref{t:CSBPCont} could serve as the basis for a comprehensive
theory of long-time scaling limits of critical CSBP, but such a study
is beyond the scope of the present paper.
A large number of results exist in the classical literature that cover 
supercritical and subcritical cases; see~\cite{Grey74,VatutinZubkov85,VatutinZubkov93,AthreyaNey04,Lambert07,Kyprianou14}, e.g., for further details.
For relatively recent results on critical cases we refer to \cite{Pakes08,Pakes10,IyerLegerEA15}.
Our paper \cite{IyerLegerEA15} provides necessary and sufficient criteria for 
approach to self-similar form for critical CSBP that become extinct 
almost surely (having branching mechanism satisfying Grey's condition),
under a quite general assumption on the scaling that is taken,
but assuming there is a unique limit as $t\to\infty$.

What we will point out here, however, is that the existence of certain
\emph{universal} critical CSBPs is now a simple consequence of our study
of universal L\'evy triples and Galton-Watson family size distributions in Section~\ref{s:Univ}.
This observation gives a precise meaning to a remark made by Grey~\cite{Grey74} 
to the effect that a large class of 
``critical and subcritical processes $\ldots$
do not seem to lend themselves to suitable scaling'' 
which yields a well-defined limit.

\begin{theorem}\label{t:UniversalCSBP}
There exists a finite critical CSBP $Z_\star$ that is universal in the following sense:
There exist sequences $(b_k), (c_k) \to \infty$ such that for any finite (sub)critical branching process $\tilde Z$
there exists a subsequence along which the finite dimensional distributions of $Z_\star^{b_k, c_k}$ converge to those of $\tilde Z$.
Here $Z_\star^{b_k, c_k}$ is the rescaled process defined by
\begin{equation*}
  Z^{b_k, c_k}_\star(x, t) \defeq b_k\inv Z_\star( b_k x, c_k t ) \,.
\end{equation*}
\end{theorem}

\begin{proof}[Proof of Theorem~\ref{t:UniversalCSBP}]
  Using Proposition~\ref{p:UniversalLT} choose $\lambda_\star=(0,0,\mu_\star)$ 
  to be a universal L\'evy triple with sequences $b_k,c_k\to\infty$.
  Let $Z_\star$ be the critical CSBP generated by $\lambda_\star$.
  The theorem now follows immediately from the continuity theorem (Theorem~\ref{t:CSBPCont}).
\end{proof}

\begin{remark*}
Let $\varphi_\star$ and $\tilde \varphi$ be the Laplace exponents of $Z_\star$ and $\tilde Z$ respectively.
Then, along the subsequence for which the above convergence holds we have
\[
b_k\, \vp_\star(b_k\inv q,c_kt) \to \tilde\vp(q,t) \quad\text{for all $q,t\in [0,\infty)$}.
\]
For each $t>0$, the functions $\vp_\star(\cdot,t)$ and $\tilde\vp(\cdot,t)$
are the Bernstein transforms of respective L\'evy triples of the form
$(\beta_{\star0}(t),0,\nu_{\star,t})$ and $(\tilde\beta_0(t),0,\tilde\nu_t)$, 
with 
\[
\vp_\star(q,t)=\beta_{\star0}(t)q+\int_0^\infty (1-e^{-qx})d\nu_{\star,t}(x)
\]
and a similar expression for $\tilde\vp$, cf.~\eqref{e:PhiBern1}.
We may note that 
\[
\beta_{\star0}(t)=\exp(-\Psi_\star'(\infty)t)\,,
\qquad
\Psi_\star'(\infty) = \int_0^\infty d\mu_\star(x) <\infty\,,
\]
and that $\beta_{\star0}(c_kt)\to0$ as $k\to\infty$.
Due to the continuity theorem~\ref{t:Cont}, the convergence above corresponds to 
a L\'evy-convergence property: for each $t>0$, as $k\to\infty$,
\[
(x\wedge1) b_k d\nu_{\star,c_k t}(b_k x) \to \beta_0(t)\delta_0+(x\wedge1)d\tilde\nu_t(x)
\quad\mbox{weak-$\star$ on $[0,\infty]$.}
\]
\end{remark*}



%
%

%
%
%
  \appendix
  \section{Continuity theorems for Bernstein transforms of L\'evy triples.}\label{s:CT}

  The L\'evy-convergence requirement (Definition~\ref{d:Lconv}) that is used in all our results relates to a natural topology of L\'evy triples associated with subordinators.
  The purpose of this appendix is to expand on this and prove a couple of continuity theorems relating convergence of L\'evy triples to pointwise convergence of the associated Bernstein functions.
  These theorems are variants of the classical continuity theorem for Laplace transforms, but do not appear to be widely known.

  
  \begin{definition}
    We say $(a_0, a_\infty, \mu)$ is a \emph{L\'evy triple} if $\mu$ is a (nonnegative) measure on $\R^+=(0,\infty)$ and we have
    \begin{equation}\label{e:LevyTripFiniteness}
      a_0 \geq 0 \,,
      \quad a_\infty \geq 0 \,,
      \quad\text{and}\quad
      \int_{\R^+} ( x \varmin 1) \, d\mu(x) < \infty \,.
    \end{equation}
    A measure $\mu$ satisfying the above is called a \emph{L\'evy measure}.
  \end{definition}

For the next definition, recall that a smooth function $g: \R^+ \to \R$ is said to be \emph{completely monotone} if $(-1)^n g^{(n)} \geq 0$ for all integer $n\geq0$.  

\begin{definition}
A function $f:\R^+\to\R$ is  \emph{Bernstein}  if it is smooth, nonnegative, and
its derivative $f'$ is completely monotone.
\end{definition}

The recent book of Schilling et al.~\cite{SchillingSongEA10} develops the theory of Bernstein functions extensively. The main representation theorem regarding these functions
(see Theorem~3.2 of \cite{SchillingSongEA10})
is the following variant of Bernstein's theorem (which states that $g$ is completely monotone if and only if it is the Laplace transform of some Radon measure on $[0,\infty)$).  

\begin{theorem}\label{t:BernsteinRep}
A function $f\colon\R^+\to\R$ is Bernstein if and only if it has the representation
  \begin{equation}\label{e:BernsteinDef}
    f(q)
      = a_0 q + a_\infty
	+ \int_{\R^+} \paren[\big]{1 - e^{-qx}} \, d\mu(x) \,,
  \end{equation}
  for some L\'evy triple $(a_0,a_\infty,\mu)$. In particular, the triple $(a_0,a_\infty,\mu)$
  determines $f$ uniquely and vice versa.
\end{theorem}

For convenience, we call the function~$f$ in~\eqref{e:BernsteinDef} the \emph{Bernstein transform} of the L\'evy triple $(a_0, a_\infty, \mu)$. If $a_0=a_\infty=0$, we call $f$
the Bernstein transform of $\mu$.

  L\'evy triples of the  form above arise naturally in the study of subordinators---right continuous, increasing (possibly infinite) processes  that have independent, time-homogeneous increments.
  We know (see for instance~\cite{Bertoin00a}) that the Laplace exponent of a subordinator can be uniquely expressed in the form \eqref{e:BernsteinDef}, as the Bernstein transform of 
  some L\'evy triple.
  In terms of the subordinator, $a_0$ represents the drift, $a_\infty$ the killing, and $\mu$ the jumps.%
  \ifdraft
\sidenote{Tue 06/16 GI: For the `usual' L\'evy triples in \eqref{e:PsiLim}, 
we have diffusion ($\alpha_0$), drift ($\alpha_\infty$), and jumps ($\mu$); how do these relate?
 
 RP: Nice Q. For a subordinator, $a_0$ is drift, $a_\infty$ is killing, and $\mu$ is jumps.
 But is there a deeper relation between subordinators and spectrally positive L\'evy processes? Dunno. \textcolor{red}{(Bob also promises not to delay publication by 3 years because of this point.)}}
  \fi

  
  We obtain a natural topology on L\'evy triples by associating to each L\'evy triple
  a finite measure on the compactified half-line $[0, \infty]$, including atoms at $0$ and $\infty$.
  Explicitly, to any L\'evy triple~$(a_0, a_\infty, \mu)$ we associate the finite measure $\kappa$ on $[0,\infty]$ defined by
  \begin{equation}\label{def:kappa}
    d\kappa(x) = a_0 \, d\delta_0(x) + a_\infty \, d\delta_\infty(x) + (x \varmin 1) \, d\mu(x) \,.
  \end{equation}
  We note that this association of finite measures with  L\'evy triples is bijective. 
  Now we use the weak-$\star$ topology of finite measures on $[0, \infty]$ to induce a topology on the set of L\'evy triples.  Note that for any $g\in C([0,\infty])$, we have
  \[
  \ip{g}{\kappa} = 
        g(0) a_0 + g(\infty) a_\infty + \int_{\R^+} g(x) ( x \varmin 1) \, d\mu(x) \,.
        \]
  \begin{definition}\label{d:LtConv}
    We say a sequence of L\'evy triples~$(a_0^{(k)}, a_\infty^{(k)}, \mu_k)$ converges to the L\'evy triple~$(a_0, a_\infty, \mu)$ if the corresponding sequence of measures $\kappa_k$  converges to $\kappa$ weak-$\star$ on $[0,\infty]$.
    That is, for every $g \in C([0, \infty])$ we have
    \[
    \ip{g}{\kappa_k} \to \ip{g}{\kappa} \quad\text{as } k\to\infty \,.
    \]
  \end{definition}
  \begin{remark}\label{r:Lconv}
  This is exactly the same as Definition~\ref{d:LtConv0}, restated here for convenience.
  Moreover, this generalizes the notion of L\'evy-convergence introduced in Definition~\ref{d:Lconv}.
  Indeed, L\'evy-convergence of $(\scalek\mu)$ to a measure $\kappa$ is exactly convergence of the L\'evy triples $(0, 0, \scalek\mu)$ to the associated L\'evy triple $(\kappa(0), \kappa(\infty), \kappa|_{(0, \infty)} )$.
\end{remark}

The following provides a continuity theorem for the Bernstein transform
that is not present in \cite{SchillingSongEA10}. It may be inferred from the proof of 
Theorem 3.1 in \cite{MenonPego08}, but \cite{MenonPego08} employs a different and rather cumbersome terminology, and the proof below is much simpler.

  \begin{theorem}\label{t:Cont}
    Let $(a_0^{(k)}, a_\infty^{(k)}, \mu_k)$ be a sequence of L\'evy triples
    with Bernstein transforms $f_k$ corresponding via \eqref{e:BernsteinDef} and
    measures $\kappa_k$ corresponding via \eqref{def:kappa}.  Then the following are equivalent.
    \begin{itemize}
\item[(i)]  $f(q) \defeq \lim_{k\to\infty} f_k(q)$ exists for each $q\in(0,\infty)$.
\item[(ii)]  $\kappa_k$ converges to some finite measure $\kappa$ weak-$\star$ on $[0,\infty]$.
That is, as $k\to\infty$,
\[
\ip{g}{\kappa_k} \to \ip{g}{\kappa} \quad\text{for every $g\in C([0,\infty])$}\,.
\]
        \end{itemize}
        If either condition holds, then the limits $f$, $\kappa$ correspond to 
        a unique L\'evy triple via \eqref{e:BernsteinDef} and \eqref{def:kappa} respectively.
  \end{theorem}
By this result, the sequence $(f_k)$ converges pointwise on $\R^+$ if and only if the sequence of L\'evy triples~$(a_0^{(k)}, a_\infty^{(k)}, \mu_k)$ 
converges in the sense of Definition~\ref{d:LtConv}. 
  \begin{proof}
    Let $(\kappa_k)$ be the sequence of measures associated with the L\'evy triples $(a_0^{(k)}, a_\infty^{(k)}, \mu_k )$ as in~\eqref{def:kappa}.
    Suppose first that $\kappa_k\to \kappa$ weak-$\star$ on $[0, \infty]$.
    For $q \in \R^+$, consider the test function $g_q:\R^+ \to \R$ defined by
    \begin{equation}\label{d:gq}
      g_q(x) = \frac{ 1 - e^{-qx} }{ x \varmin 1 } \,.
    \end{equation}
    By defining $g_q(0) = q$ and $g_q(\infty) = 1$ we can extend $g_q$ to a continuous function on $[0, \infty]$.
    Consequently,
    \begin{equation*}
      \ip{g_q}{\kappa}
	= \lim_{k \to \infty} \ip{g_q}{\kappa_k}
	= \lim_{k \to \infty}
	    a_0^{(k)} q + a_{\vphantom{0}\infty}^{(k)}
	    + \int_{\R^+} \paren[\big]{1 - e^{-qx}} \, d\mu(x)
	= \lim_{k \to \infty} f_k(q) \,,
    \end{equation*}
    establishing pointwise convergence of $(f_k)$ on $\R^+$ as desired.

    For the converse, suppose $(f_k) \to f$ pointwise on $\R^+$.
    Note that for $q \in \R^+$, we have
    \begin{equation*}
      c_q \defeq \inf_{x \in \R^+} g_q(x) > 0 \,.
    \end{equation*}
    Consequently,
    \begin{equation*}
      \sup_k{} \ip{1}{\kappa_k}
	\leq \sup_k \frac{1}{c_q} \ip{g_q}{\kappa_k}
	= \frac{1}{c_q} \sup_k f_k(q)
	< \infty \,.
    \end{equation*}
    Thus, by the Banach-Alaoglu theorem, any subsequence of $(\kappa_k)$ has a 
    further subsequence that is weak-$\star$ convergent on $[0, \infty]$. 
    Let $\kappa$ denote any such subsequential limit. 
By taking limits as above but along subsequences, we infer that for every $q\in\R^+$,
\[
\ip{g_q}{\kappa} = f(q) \,.
\]
This shows that $f$ is the Bernstein transform of the L\'evy triple associated to
$\kappa$ by \eqref{def:kappa}. Because both this association and the Bernstein transform
are bijective, $\kappa$ is uniquely determined by $f$. It follows that the entire sequence
$(\kappa_k)$ converges weak-$\star$ to the same limit $\kappa$. 
%
  \end{proof}
  
We finish by developing two further variations of the convergence conditions in the continuity theorem above.
To each L\'evy triple $(a_0,a_\infty,\mu)$, we associate two further quantities:
(a) The \emph{Bernstein primitive} (or branching mechanism)
\begin{equation} \label{def:Bprimitive}
\psi(q) = \frac12 a_0 q^2 + a_\infty q + \int_{\R^+} 
\frac{e^{-qx}-1+qx}x \,d\mu(x) \,,
\end{equation}
and (b) the pair of \emph{left and right distribution functions} $(\kappa_L,\kappa_R)$ given by
\begin{equation} \label{def:kappaLR}
\kappa_L(x) = a_0+\int_{(0,x]}z\,d\mu(z) \,, \quad
\kappa_R(x) = a_\infty + \int_{(x,\infty)}\,d\mu(z) \,.
\end{equation}
These distribution functions are associated to Radon measures, also denoted by
$\kappa_L$ on $[0,\infty)$ and $\kappa_R$ on $(0,\infty]$, in a standard way.

\begin{theorem}\label{t:LRconv}
Make the same assumptions as in Theorem~\ref{t:Cont}, and for each $k$
associate $\psi_k$ and $(\kappa_{L,k},\kappa_{R,k})$ via \eqref{def:Bprimitive} and \eqref{def:kappaLR} respectively.
Then the following conditions are also equivalent to (i) and (ii).
\begin{itemize}
\item[(iii)] $\psi(q) \defeq \lim_{k\to\infty}\psi_k(q)$ exists for each $q\in(0,\infty)$. 
\item[(iv)] For almost every $x\in(0,\infty)$, both of the following limits exist:
\begin{equation}
\kappa_L(x) = \lim_{k\to\infty} \kappa_{L,k}(x)\,, \quad
\kappa_R(x) = \lim_{k\to\infty} \kappa_{R,k}(x)\,.
\end{equation}
\end{itemize}
\end{theorem}

\begin{remark*}
The condition in (iv) is equivalent to weak-$\star$ convergence of
the measures $\kappa_{L,k}$ and $\kappa_{R,k}$ on the intervals 
$[0,\infty)$ and $(0,\infty]$ respectively. 
\end{remark*}

\begin{proof}
By Theorem~\ref{t:Cont}, it suffices to show (i) is equivalent to (iii), and (ii) is equivalent to (iv).

Suppose (i) holds. Note that $f_k$ is increasing and concave for all $k$, hence for each $q\in\R^+$, by dominated convergence we have
\[
\psi_k(q) = \int_0^q f_k(r)\,dr \to \int_0^q f(r)\,dr \eqdef \psi(q) \,.
\]
Hence (iii) holds. Conversely, suppose (iii) holds. Note that $\psi_k(2q)\geq q f_k(q)$ for all $q>0$. For each $q>0$, it follows $(f_k(q))$ is precompact. 
Because $f_k$ is increasing and concave, each subsequence has a further 
subsequence that converges pointwise to some (increasing and concave) function $f$.
Necessarily,
\[
\int_0^q f(r)\,dr = \psi(q) \quad\text{for each $q>0$} \,,
\]
therefore $\psi$ is $C^1$ and $\psi'=f$. Thus $f$ is determined by $\psi$,
and it follows that the whole sequence $(f_k)$ converges, proving (i). 

Next we prove (ii) implies (iv). Given $\kappa$ as in (ii) we may define 
$\kappa_L,\kappa_R$ so that 
\begin{equation}
d\kappa_L(x) = \frac{x\, d\kappa(x)}{x\wedge 1} \,, \quad
d\kappa_R(x) = \frac{ d\kappa(x)}{x\wedge 1}\,,
\end{equation}
on the intervals $[0,\infty)$, $(0,\infty]$ respectively.
For any $g_0\in C_c([0,\infty))$ and $g_\infty\in C_c((0,\infty])$, let 
\begin{equation}
g_L(x) = \frac{x\, g_0(x)}{x\wedge 1}\,, \quad
g_R(x) = \frac{g_\infty(x)}{x\wedge 1}\,.
\label{def:gLR}
\end{equation}
Then as $k\to\infty$ we have 
\begin{align}
&\ip{g_0}{\kappa_{L,k}} = \ip{g_L}{\kappa_k} \to \ip{g_L}{\kappa} \eqdef \ip{g_0}{\kappa_L} \,,
\\
&\ip{g_\infty}{\kappa_{R,k}} = \ip{g_R}{\kappa_k}\to\ip{g_R}{\kappa} \eqdef \ip{g_\infty}{\kappa_R} \,.
\end{align}
This establishes weak-$\star$ convergence of the measures
$\kappa_{L,k}$ to $\kappa_L$ and $\kappa_{R,k}$ to $\kappa_R$, hence (iv) holds.

Conversely, assume (iv).  Fix smooth cutoff functions $v_L, v_R:[0,\infty]\to[0,1]$ such
that $v_L+v_R=1$, $v_L=1$ on $[0,1]$ and $v_R=1$ on $[2,\infty]$. Given any 
$g\in C([0,\infty])$ we write $g=g_L+g_R$ where $g_L=gv_L $, $g_R=gv_R $,
and determine $g_0$, $g_\infty$ by \eqref{def:gLR}.
Then 
\begin{equation}
\ip{g_L+g_R}{\kappa_k} = \ip{g_0}{\kappa_{L,k}}+\ip{g_\infty}{\kappa_{R,k}}
\to \ip{g_0}{\kappa_L}+\ip{g_\infty}{\kappa_R}
\end{equation}
as $k\to\infty$. Because $\lim_{k\to\infty}\ip{g}{\kappa_k}$  exists for each $g$, there exists 
a measure $\kappa$ on $[0,\infty]$ with $\kappa_k\to\kappa$ weak-$\star$ on $[0,\infty]$.
Thus (ii) holds.
\end{proof}

%
%
%
  \bibliographystyle{halpha-abbrv}
  \bibliography{refs}

\newcommand{\etalchar}[1]{$^{#1}$}
\def\cprime{$'$}
\begin{thebibliography}{BBC{\etalchar{+}}05}
\expandafter\ifx\csname url\endcsname\relax
  \def\url#1{\texttt{#1}}\fi
\expandafter\ifx\csname doi\endcsname\relax
  \def\doi#1{\burlalt{doi:#1}{http://dx.doi.org/#1}}\fi
\expandafter\ifx\csname urlprefix\endcsname\relax\def\urlprefix{URL }\fi
\expandafter\ifx\csname href\endcsname\relax
  \def\href#1#2{#2}\fi
\expandafter\ifx\csname burlalt\endcsname\relax
  \def\burlalt#1#2{\href{#2}{#1}}\fi

\bibitem[Ald99]{Aldous99}
D.~J. Aldous.
\newblock Deterministic and stochastic models for coalescence (aggregation and
  coagulation): a review of the mean-field theory for probabilists.
\newblock {\em Bernoulli}, 5(1):3--48, 1999.
\newblock \doi{10.2307/3318611}.

\bibitem[AN04]{AthreyaNey04}
K.~B. Athreya and P.~E. Ney.
\newblock {\em Branching processes}.
\newblock Dover Publications, Inc., Mineola, NY, 2004.
\newblock Reprint of the 1972 original [Springer, New York; MR0373040].

\bibitem[Bac11]{Bacaer11}
N.~Baca{\"e}r.
\newblock {\em A short history of mathematical population dynamics}.
\newblock Springer-Verlag London, Ltd., London, 2011.
\newblock \doi{10.1007/978-0-85729-115-8}.

\bibitem[BBC{\etalchar{+}}05]{BirknerBlathEA05}
M.~Birkner, J.~Blath, M.~Capaldo, A.~Etheridge, M.~M{\"o}hle, J.~Schweinsberg,
  and A.~Wakolbinger.
\newblock Alpha-stable branching and beta-coalescents.
\newblock {\em Electron. J. Probab.}, 10:no. 9, 303--325, 2005.
\newblock \doi{10.1214/EJP.v10-241}.

\bibitem[BBL14]{BerestyckiBerestyckiEA14}
J.~Berestycki, N.~Berestycki, and V.~Limic.
\newblock A small-time coupling between {$\Lambda$}-coalescents and branching
  processes.
\newblock {\em Ann. Appl. Probab.}, 24(2):449--475, 2014.
\newblock \doi{10.1214/12-AAP911}.

\bibitem[BBS08]{BerestyckiBerestyckiEA08}
J.~Berestycki, N.~Berestycki, and J.~Schweinsberg.
\newblock Small-time behavior of beta coalescents.
\newblock {\em Ann. Inst. Henri Poincar\'e Probab. Stat.}, 44(2):214--238,
  2008.
\newblock \doi{10.1214/07-AIHP103}.

\bibitem[Ber00]{Bertoin00a}
J.~Bertoin.
\newblock Subordinators, {L}\'evy processes with no negative jumps, and
  branching processes, 2000.

\bibitem[Ber06]{Bertoin06}
J.~Bertoin.
\newblock {\em Random fragmentation and coagulation processes}, volume 102 of
  {\em Cambridge Studies in Advanced Mathematics}.
\newblock Cambridge University Press, Cambridge, 2006.
\newblock \doi{10.1017/CBO9780511617768}.

\bibitem[Ber09]{Berestycki09}
N.~Berestycki.
\newblock {\em Recent progress in coalescent theory}, volume~16 of {\em Ensaios
  Matem\'aticos [Mathematical Surveys]}.
\newblock Sociedade Brasileira de Matem\'atica, Rio de Janeiro, 2009.

\bibitem[BLG00]{BertoinLeGall00}
J.~Bertoin and J.-F. Le~Gall.
\newblock The {B}olthausen-{S}znitman coalescent and the genealogy of
  continuous-state branching processes.
\newblock {\em Probab. Theory Related Fields}, 117(2):249--266, 2000.
\newblock \doi{10.1007/s004400050006}.

\bibitem[BLG03]{BertoinLeGall03}
J.~Bertoin and J.-F. Le~Gall.
\newblock Stochastic flows associated to coalescent processes.
\newblock {\em Probab. Theory Related Fields}, 126(2):261--288, 2003.
\newblock \doi{10.1007/s00440-003-0264-4}.

\bibitem[BLG05]{BertoinLeGall05}
J.~Bertoin and J.-F. Le~Gall.
\newblock Stochastic flows associated to coalescent processes. {II}.
  {S}tochastic differential equations.
\newblock {\em Ann. Inst. H. Poincar\'e Probab. Statist.}, 41(3):307--333,
  2005.
\newblock \doi{10.1016/j.anihpb.2004.07.003}.

\bibitem[BLG06]{BertoinLeGall06}
J.~Bertoin and J.-F. Le~Gall.
\newblock Stochastic flows associated to coalescent processes. {III}. {L}imit
  theorems.
\newblock {\em Illinois J. Math.}, 50(1-4):147--181 (electronic), 2006.
\newblock \urlprefix\url{http://projecteuclid.org/euclid.ijm/1258059473}.

\bibitem[BS15]{BansayeSimatos15}
V.~Bansaye and F.~Simatos.
\newblock On the scaling limits of {G}alton-{W}atson processes in varying
  environments.
\newblock {\em Electron. J. Probab.}, 20:no. 75, 36, 2015.
\newblock \doi{10.1214/EJP.v20-3812}.

\bibitem[CLUB09]{CaballeroLambertEA09}
M.~E. Caballero, A.~Lambert, and G.~Uribe~Bravo.
\newblock Proof(s) of the {L}amperti representation of continuous-state
  branching processes.
\newblock {\em Probab. Surv.}, 6:62--89, 2009.
\newblock \doi{10.1214/09-PS154}.

\bibitem[Fel71]{Feller71}
W.~Feller.
\newblock {\em An introduction to probability theory and its applications.
  {V}ol. {II}.}
\newblock Second edition. John Wiley \& Sons Inc., New York, 1971.

\bibitem[GH16]{GrosjeanHuillet16}
N.~Grosjean and T.~Huillet.
\newblock On a coalescence process and its branching genealogy.
\newblock {\em J. Appl. Probab.}, 53(4):1156--1165, 2016.
\newblock \doi{10.1017/jpr.2016.71}.

\bibitem[GIM14]{GnedinIksanovEA14}
A.~Gnedin, A.~Iksanov, and A.~Marynych.
\newblock {$\Lambda$}-coalescents: a survey.
\newblock {\em J. Appl. Probab.}, 51A(Celebrating 50 Years of The Applied
  Probability Trust):23--40, 2014.
\newblock \doi{10.1239/jap/1417528464}.

\bibitem[Gre74]{Grey74}
D.~R. Grey.
\newblock Asymptotic behaviour of continuous time, continuous state-space
  branching processes.
\newblock {\em J. Appl. Probability}, 11:669--677, 1974.

\bibitem[Gri74]{Grimvall74}
A.~Grimvall.
\newblock On the convergence of sequences of branching processes.
\newblock {\em Ann. Probability}, 2:1027--1045, 1974.

\bibitem[HS77]{HeydeSeneta77}
C.~C. Heyde and E.~Seneta.
\newblock {\em I. {J}. {B}ienaym\'e. {S}tatistical theory anticipated}.
\newblock Springer-Verlag, New York-Heidelberg, 1977.
\newblock Studies in the History of Mathematics and Physical Sciences, No. 3.

\bibitem[ILP15]{IyerLegerEA15}
G.~Iyer, N.~Leger, and R.~L. Pego.
\newblock Limit theorems for {S}moluchowski dynamics associated with critical
  continuous-state branching processes.
\newblock {\em Ann. Appl. Probab.}, 25(2):675--713, 2015.
\newblock \doi{10.1214/14-AAP1008}.

\bibitem[Kin82a]{Kingman82}
J.~F.~C. Kingman.
\newblock The coalescent.
\newblock {\em Stochastic Process. Appl.}, 13(3):235--248, 1982.
\newblock \doi{10.1016/0304-4149(82)90011-4}.

\bibitem[Kin82b]{Kingman82a}
J.~F.~C. Kingman.
\newblock On the genealogy of large populations.
\newblock {\em J. Appl. Probab.}, (Special Vol. 19A):27--43, 1982.
\newblock Essays in statistical science.

\bibitem[Kyp14]{Kyprianou14}
A.~E. Kyprianou.
\newblock {\em Fluctuations of {L}\'evy processes with applications}.
\newblock Universitext. Springer, Heidelberg, second edition, 2014.
\newblock \doi{10.1007/978-3-642-37632-0}.
\newblock Introductory lectures.

\bibitem[Lam67a]{Lamperti67}
J.~Lamperti.
\newblock Continuous state branching processes.
\newblock {\em Bull. Amer. Math. Soc.}, 73:382--386, 1967.

\bibitem[Lam67b]{Lamperti67a}
J.~Lamperti.
\newblock The limit of a sequence of branching processes.
\newblock {\em Z. Wahrscheinlichkeitstheorie und Verw. Gebiete}, 7:271--288,
  1967.

\bibitem[Lam03]{Lambert03}
A.~Lambert.
\newblock Coalescence times for the branching process.
\newblock {\em Adv. in Appl. Probab.}, 35(4):1071--1089, 2003.
\newblock \doi{10.1239/aap/1067436335}.

\bibitem[Lam07]{Lambert07}
A.~Lambert.
\newblock Quasi-stationary distributions and the continuous-state branching
  process conditioned to be never extinct.
\newblock {\em Electron. J. Probab.}, 12:no. 14, 420--446, 2007.
\newblock \doi{10.1214/EJP.v12-402}.

\bibitem[Li00]{Li00}
Z.-H. Li.
\newblock Asymptotic behaviour of continuous time and state branching
  processes.
\newblock {\em J. Austral. Math. Soc. Ser. A}, 68(1):68--84, 2000.

\bibitem[LvR15]{LaurencotRoessel15}
P.~Lauren{\c{c}}ot and H.~van Roessel.
\newblock Absence of gelation and self-similar behavior for a
  coagulation-fragmentation equation.
\newblock {\em SIAM J. Math. Anal.}, 47(3):2355--2374, 2015.
\newblock \doi{10.1137/140976236}.

\bibitem[MP04]{MenonPego04}
G.~Menon and R.~L. Pego.
\newblock Approach to self-similarity in {S}moluchowski's coagulation
  equations.
\newblock {\em Comm. Pure Appl. Math.}, 57(9):1197--1232, 2004.
\newblock \doi{10.1002/cpa.3048}.

\bibitem[MP08]{MenonPego08}
G.~Menon and R.~L. Pego.
\newblock The scaling attractor and ultimate dynamics for {S}moluchowski's
  coagulation equations.
\newblock {\em J. Nonlinear Sci.}, 18(2):143--190, 2008.
\newblock \doi{10.1007/s00332-007-9007-5}.

\bibitem[Pak08]{Pakes08}
A.~G. Pakes.
\newblock Conditional limit theorems for continuous time and state branching
  process.
\newblock In M.~Ahsanullah and G.~P. Yanev, editors, {\em Records and Branching
  Processes}, pages 63--103. Nova Science Publishers, Inc., 2008.

\bibitem[Pak10]{Pakes10}
A.~G. Pakes.
\newblock Critical {M}arkov branching process limit theorems allowing infinite
  variance.
\newblock {\em Adv. in Appl. Probab.}, 42(2):460--488, 2010.
\newblock \doi{10.1239/aap/1275055238}.

\bibitem[Pit99]{Pitman99a}
J.~Pitman.
\newblock Coalescents with multiple collisions.
\newblock {\em Ann. Probab.}, 27(4):1870--1902, 1999.

\bibitem[Sag99]{Sagitov99}
S.~Sagitov.
\newblock The general coalescent with asynchronous mergers of ancestral lines.
\newblock {\em J. Appl. Probab.}, 36(4):1116--1125, 1999.

\bibitem[Sch03]{Schweinsberg03}
J.~Schweinsberg.
\newblock Coalescent processes obtained from supercritical {G}alton-{W}atson
  processes.
\newblock {\em Stochastic Process. Appl.}, 106(1):107--139, 2003.

\bibitem[SSV10]{SchillingSongEA10}
R.~L. Schilling, R.~Song, and Z.~Vondra{\v{c}}ek.
\newblock {\em Bernstein functions}, volume~37 of {\em de Gruyter Studies in
  Mathematics}.
\newblock Walter de Gruyter \& Co., Berlin, 2010.
\newblock Theory and applications.

\bibitem[vS16]{Smoluchowski16}
M.~von Smoluchowski.
\newblock Drei {V}ortr{\"a}ge {\"u}ber {D}iffusion, {B}rownsche {B}ewegung und
  {K}oagulation von {K}olloidteilchen.
\newblock {\em Physik. Z.}, 17:557--585, 1916.

\bibitem[vS17]{Smoluchowski17}
M.~von Smoluchowski.
\newblock Experiments on a mathematical theory of kinetic coagulation of coloid
  solutions.
\newblock {\em Zeitschrift fur Physikalische Chemie--Stochiometrie und
  Verwandtschaftslehre}, 92(2):129--168, 1917.

\bibitem[VZ85]{VatutinZubkov85}
V.~A. Vatutin and A.~M. Zubkov.
\newblock Branching processes. {I}.
\newblock In {\em Probability theory. {M}athematical statistics. {T}heoretical
  cybernetics, {V}ol.\ 23}, Itogi Nauki i Tekhniki, pages 3--67, 154. Akad.
  Nauk SSSR, Vsesoyuz. Inst. Nauchn. i Tekhn. Inform., Moscow, 1985.

\bibitem[VZ93]{VatutinZubkov93}
V.~A. Vatutin and A.~M. Zubkov.
\newblock Branching processes. {II}.
\newblock {\em J. Soviet Math.}, 67(6):3407--3485, 1993.
\newblock \doi{10.1007/BF01096272}.
\newblock Probability theory and mathematical statistics, 1.

\bibitem[WG75]{WatsonGalton75}
H.~W. Watson and F.~Galton.
\newblock On the probability of the extinction of families.
\newblock {\em The Journal of the Anthropological Institute of Great Britain
  and Ireland}, 4:138--144, 1875.

\end{thebibliography}

\end{document}